\documentclass[12pt]{amsart}

\usepackage{amsmath,amssymb,amsthm,graphicx,color,amscd}
\usepackage[usenames,dvipsnames]{xcolor}
\usepackage{enumitem}

\usepackage[small,nohug,heads=vee]{diagrams}

\usepackage{hyperref}

\numberwithin{equation}{section}
\hypersetup{bookmarksdepth=3}

\theoremstyle{plain}
\newtheorem{theorem}[equation]{Theorem}

\newtheorem{conjecture}[equation]{Conjecture}
\newtheorem{regularity_conjecture}[equation]{Regularity Conjecture}

\newtheorem*{claim}{Claim}
\newtheorem{proposition}[equation]{Proposition}
\newtheorem{lemma}[equation]{Lemma}

\theoremstyle{definition}
\newtheorem{definition}[equation]{Definition}
\newtheorem{example}[equation]{Example}

\theoremstyle{remark}
\newtheorem{remark}[equation]{Remark}

\newcommand{\acts}{\operatorname{\curvearrowright}}
\newcommand{\ad}{\operatorname{ad}}

\newcommand{\al}{\alpha}

\newcommand{\aut}{\operatorname{Aut}}

\newcommand{\be}{\beta}
\newcommand{\ben}{\begin{enumerate}}
\newcommand{\bit}{\begin{itemize}}
\newcommand{\C}{\mathbb{C}}

\newcommand{\cof}{\operatorname{cof}}

\newcommand{\de}{\delta}

\newcommand{\Div}{\operatorname{div}}
\newcommand{\diam}{\operatorname{diam}}

\newcommand{\een}{\end{enumerate}}
\newcommand{\eit}{\end{itemize}}

\newcommand{\eps}{\varepsilon}

\newcommand{\fg}{\mathfrak{g}}
\newcommand{\fh}{\mathfrak{h}}

\newcommand{\F}{\mathbb{F}}

\newcommand{\ga}{\gamma}
\newcommand{\Ga}{\Gamma}

\renewcommand{\H}{\mathbb{H}}

\newcommand{\id}{\operatorname{id}}

\newcommand{\imag}{\operatorname{\frak{Im}}}

\newcommand{\la}{\lambda}
\newcommand{\La}{\Lambda}

\newcommand{\loc}{\operatorname{loc}}
\newcommand{\lra}{\longrightarrow}

\newcommand{\ol}{\overline}
\newcommand{\om}{\omega}
\newcommand{\Om}{\Omega}

\newcommand{\R}{\mathbb{R}}
\newcommand{\ra}{\rightarrow}

\newcommand{\rank}{\operatorname{rank}}
\definecolor{gray}{gray}{0.7}

\newcommand{\real}{\operatorname{\frak{Re}}}

\newcommand{\restr}{\mbox{\Large \(|\)\normalsize}}

\newcommand{\si}{\sigma}

\newcommand{\Span}{\operatorname{span}}

\newcommand{\stab}{\operatorname{Stab}}
\newcommand{\sympl}{\operatorname{Sym}}

\renewcommand{\th}{\theta}

\newcommand{\we}{\wedge}

\newcommand{\wt}{\operatorname{wt}}
\newcommand{\Z}{\mathbb{Z}}

\def\XXint#1#2#3{{\setbox0=\hbox{$#1{#2#3}{\int}$ }
\vcenter{\hbox{$#2#3$ }}\kern-.6\wd0}}

\begin{document}

\title{Sobolev mappings between nonrigid Carnot groups }

\author{Bruce Kleiner}
\thanks{BK was supported by NSF grants DMS-1711556 and DMS-2005553, and a Simons Collaboration grant.}
\email{bkleiner@cims.nyu.edu}
\address{Courant Institute of Mathematical Science, New York University, 251 Mercer Street, New York, NY 10012}
\author{Stefan M\"uller}
\thanks{SM has been supported by the Deutsche Forschungsgemeinschaft (DFG, German Research Foundation) through
the Hausdorff Center for Mathematics (GZ EXC 59 and 2047/1, Projekt-ID 390685813) and the 
collaborative research centre  {\em The mathematics of emerging effects} (CRC 1060, Projekt-ID 211504053).  This work was initiated during a sabbatical of SM at the Courant Institute and SM would like to thank  R.V. Kohn and the Courant Institute
members and staff for 
their  hospitality and a very inspiring atmosphere.}
\email{stefan.mueller@hcm.uni-bonn.de}
\address{Hausdorff Center for Mathematics, Universit\"at Bonn, Endenicher Allee 60, 53115 Bonn}
\author{Xiangdong Xie}
\thanks{XX has been supported by Simons Foundation grant \#315130.}
\email{xiex@bgsu.edu}
\address{Dept. of Mathematics and Statistics, Bowling Green State University, Bowling Green, OH 43403}

\maketitle

\begin{abstract}
We consider mappings $f:G\supset U\ra G'$ where $G$ and $G'$ are Carnot groups.  In this paper, which is a continuation of \cite{KMX1}, we focus on Carnot groups which are nonrigid in the sense of Ottazzi-Warhurst.  We show that quasisymmetric homeomorphisms are reducible in the sense that they preserve a special type of coset foliation, unless the group is isomorphic to $\R^n$ or a real or complex Heisenberg group (where the assertion fails).  We use this to prove the quasisymmetric rigidity conjecture for such groups.  The starting point of the proof is the pullback theorem established in \cite{KMX1}.  
\end{abstract}

\tableofcontents

\section{Introduction}

\bigskip\bigskip
This is the second in a series of papers \cite{KMX1,kmx_approximation_low_p,kmx_rumin,kmx_iwasawa} on geometric mapping theory in Carnot groups, in which we establish new regularity, rigidity, and partial rigidity results for bilipschitz, quasiconformal, or more generally, Sobolev mappings, between Carnot groups.  Our motivation comes from the analytical theory of quasiconformal homeomorphisms, geometric group theory (especially rigidity phenomena and Gromov hyperbolic spaces), analysis on metric spaces, the differential geometry of subriemannian manifolds, and the literature on rigidity and oscillatory solutions to partial differential equations (and relations).

While our main interest in this paper is in quasiconformal homeomorphisms -- and more generally Sobolev mappings -- between open subsets of Carnot groups,  for context we first consider the special case of quasiconformal diffeomorphisms. 

For the remainder of the introduction $G$ will be a step $s$ Carnot group with Lie algebra $\fg$, grading $\fg=\oplus_{j=1}^sV_j$, dilation group $\{\de_r:G\ra G\}_{r\in (0,\infty)}$, and homogeneous dimension $\nu=\sum_jj\dim V_j$.  Without explicit mention, $G$ will be equipped with Haar measure and a Carnot-Caratheodory metric denoted generically by $d_{CC}$.  We will use the same notation with primes for another Carnot group $G'$.

It is easy to see that a diffeomorphism $f:U \ra U'$ between open subsets $U,U'\subset G$ is locally quasiconformal if and only if it is contact,
i.e. the differential $Df$ preserves the horizontal subbundle $V_1\subset TG$.   The study of contact diffeomorphisms has a long and fascinating history intertwined with the theory of Lie pseudogroups, overdetermined systems, and $G$-structures (in the sense of E. Cartan); 
the literature extends back to the 19th century,
    with major contributions in 1900-10 by Cartan, in 1955-70 by Kuranishi, Singer, Sternberg, Guillemin, Quillen, 
and Tanaka (among many others); there has been a resurgence of interest  in recent decades, coming from new connections with geometric group theory and quasiconformal mappings.   
The literature on this topic is substantial, so we will mention just a few points which are directly relevant to our setting, and refer the interested reader to \cite{singer_sternberg_infinite_groups_lie_cartan,tanaka_differential_systems_graded_lie_algebras,cowling_ottazzi} for references and more discussion. 
The contact condition is a nonlinear system of PDEs which is formally overdetermined except when $G$ is the Engel group or a product $\H_k\times \R^\ell$ for some $k,\ell\geq 0$, and hence one expects some form of rigidity in the generic case; here $\H_k$ denotes the $k^{th}$ Heisenberg group.   However, the analytical character of the condition is quite different for different groups: 
\bit
\item When $G=\H\times\H$ contact diffeomorphisms must be products (locally) but otherwise are quite flexible \cite{cowling_reimann_three_examples}.
\item When $G=\H_n^\C$ is the complexification of $\H_n$ the contact condition is locally flexible, but still ``hypoelliptic'', i.e. contact diffeomorphisms are holomorphic or antiholomorphic \cite{reimann_ricci}.
\item When $G$ is an $H$-type group with center of dimension at least $3$, (e.g. one of the Carnot groups studied by Pansu) or a free Carnot  group of step $s\geq 3$,  then the smooth contact embeddings $G\supset U\ra G$ form a finite dimensional family when $U$ is a connected open subset \cite{pansu,reimann_h_type,warhurst_tanaka_prolongation_free}.
\eit 
Although it has been formalized in different ways, there is a dichotomy of contact systems into finite (or rigid) type  and infinite (or nonrigid) type  according to whether the contact diffeomorphisms are locally determined by finitely many, or infinitely many parameters \cite{cartan_1904,kuranisha_local_theory_continuous_pseudogroups_i,guillemin_sternberg_algebraic_model_transitive_differential_geometry,singer_sternberg_infinite_groups_lie_cartan,tanaka_differential_systems_graded_lie_algebras}.     This was greatly clarified by Ottazzi-Warhurst \cite{ottazzi_warhurst} who proved a fundamental result characterizing rigidity/nonrigidity of the Carnot group $G$ in terms of the complexification $\fg^\C$ of the Lie algebra $\fg$, and showed that $C^2$ contact diffeomorphisms of rigid Carnot groups are $C^\infty$ (which readily implies that they are real analytic)\footnote{Recently  Jonas Lelmi improved this by replacing the $C^2$ regularity assumption with $C^1$ (or even Euclidean bilipschitz); the same result was shown 
by Alex Austin for  the $(2,3,5)$ distribution \cite{lelmi,austin_235}.}.    The prolongation theory in \cite{tanaka_differential_systems_graded_lie_algebras} gives an algebraic approach to determine the local contact diffeomorphisms,  at least in principle;  however, as it stands currently,  the general theory is based on the Cartan-K\"ahler theorem and therefore only provides a satisfactory picture for real analytic diffeomorphisms.   In particular, to our knowledge apart from some specific groups which have been analyzed \cite{reimann_ricci,warhurst_filiform,warhurst_jet_spaces}, in the nonrigid case  
there has been limited progress toward understanding $C^k$ contact diffeomorphisms, $1\leq k\leq\infty$; even a conjectural description is lacking.   

We now drop the smoothness assumption.    Motivated by the above regularity result in the smooth case, for Ottazzi-Warhurst rigid Carnot groups we have
(cf. \cite[p.2]{ottazzi_warhurst_algebraic_prolongation,ottazzi_warhurst}):
\begin{regularity_conjecture}\label{conj_regularity_conjecture}
If $G$ is a rigid Carnot group, then any quasiconformal homeomorphism $G\supset U\ra U'\subset G$ is $C^\infty$.  
\end{regularity_conjecture}
\noindent
Together with earlier work of Tanaka \cite{tanaka_differential_systems_graded_lie_algebras}, the conjecture implies a description of quasiconformal mappings in algebraic terms, at least in principle.   Conjecture~\ref{conj_regularity_conjecture} has been proven for groups whose graded automorphisms act conformally on the first layer, by subelliptic regularity \cite{capogna,capogna_cowling}; it follows from \cite{KMX1} that the  conjecture holds for a product of Carnot groups $\prod_iG_i$ if it holds for all of the factors $G_i$.  
We expect that   Conjecture~\ref{conj_regularity_conjecture} holds under weaker regularity assumptions, for instance for suitably nondegenerate $W^{1,p}_{\loc}$-mappings when $p$ is strictly larger than the homogeneous dimension $\nu$ of $G$; on the other hand it is  unclear what to expect when $p$ is small (or $p=1$), i.e. whether regularity/rigidity holds, or if such mappings are more flexible (cf. \cite[p.2]{ottazzi_warhurst}).   We will present further results on rigid groups elsewhere \cite{kmx_in_preparation}, and now turn to the case of nonrigid groups.  

We recall that a quasiconformal homeomorphism between open subsets of G is a Sobolev mapping with respect to the Carnot distance -- it
belongs to $W^{1,p}_{\loc}$ for some $p >\nu$.  More generally,  let $f:G\supset U\ra G'$ be  a $W^{1,p}_{\loc}$-mapping for some $p>\nu$, where  $U\subset G$ is open.  Then $f$  has a well-defined  Pansu differential at a.e. $x\in U$; this is a graded group homomorphism $D_Pf(x):G\ra G'$ which we often conflate with the associated homomorphism of graded Lie algebras $D_Pf(x):\fg\ra \fg'$.  If $f$ is a quasiconformal homeomorphism, then  $\det(D_Pf(x))
\neq 0$ for a.e. $x\in U$, and has the same sign as the local degree of $f$.
See  \cite{pansu}  or Theorem 2.12 in \cite{KMX1}.

Our main result is that apart from some exceptional cases, for nonrigid Carnot groups one always has partial rigidity -- mappings (virtually) preserve a foliation.

\begin{theorem}[See Section~\ref{sec_qs_rigidity_carnot_groups} for definitions]
\label{thm_nonrigid_intro}
If $G$ is a nonrigid Carnot group (in the sense of Ottazzi-Warhurst) with homogeneous dimension $\nu$, then one of the following holds:
\ben
\item $G$ is isomorphic to $\R^n$ or to a real or complex Heisenberg group $\H_n$, $\H_n^\C$ for some $n\geq 1$.
\item There is a closed horizontally generated subgroup $\{e\}\subsetneq H\subsetneq G$, a constant $K$, and a finite    set   $A$ of graded automorphisms of $G$ with the following properties:    
\bit
\item For every $p>\nu$, $x\in G$, $r\in(0,\infty)$, and every $W^{1,p}_{\loc}$-mapping $$f:G\supset B(x,r)\ra G$$ such that the sign of $\det(D_Pf)$ is constant almost everywhere, then for some $\Phi\in A$ the restriction of the composition $\Phi\circ f$ to the subball $B(x,\frac{r}{K})$ preserves the coset foliation of $H$. In particular, the conclusion holds for quasiconformal homeomorphisms. 
\item The Lie algebra of $H$ is generated by a linear subspace $\{0\}\subsetneq W\subsetneq V_1$ with $[W,V_j]=\{0\}$ for all $j\geq 2$.
\eit
\een
\end{theorem}
Thus, apart from the exceptional cases in (1), quasiconformal homeomorphisms  preserve a foliation, up to post-composition with a graded automorphism.   This suggests that for the (nonexceptional) nonrigid cases there may be a more detailed description of quasiconformal homeomorphisms along the lines of \cite{xie_filiform} for model filiform groups.  To our knowledge such a description is not known even for smooth contact diffeomorphisms, either locally or globally.

\begin{remark}
\label{rem_better_sobolev_exponent}
Using more refined version of the pullback theorem from \cite{kmx_approximation_low_p} rather than the version in \cite{KMX1} applied here, in some cases one can reduce the requirements on the Sobolev exponent $p$ in Theorem~\ref{thm_nonrigid_intro}.
\end{remark}
Another aspect of quasiconformal rigidity/flexibility  has to do with the Sobolev exponent, i.e. higher integrability of the derivative.  Quasiconformal homeomorphisms $f:G\supset U\ra U'\subset G$ are always in $W^{1,p}_{\loc}$ for some $p$ strictly larger than the homogeneous dimension of $G$, where $p$  depends on $G$ and the quasiconformal distortion of $f$ \cite{heinonen_koskela}.  However, except for $\R^n$ and the Heisenberg groups $\H_n$ (see for example \cite{balogh_non_bilipschitz}), all known examples of quasiconformal homeomorphisms are in $W^{1,\infty}_{\loc}$, i.e. are locally bilipschitz.  This inspired the following:
\begin{conjecture}[Xie]
\label{conj_qs_rigidity}
If $G$ is a Carnot group other than $\R^n$ or $\H_n$ for some $n$, then every quasiconformal homeomorphism $f:G\supset U\ra U'\subset G$ is locally bilipschitz; moreover, if $U=G$ then $f$ is bilipschitz.
\end{conjecture}
Using Theorem~\ref{thm_nonrigid_intro} and a variation on \cite{Shan_Xie, Xie_Pacific2013,LeDonne_Xie} we show:
\begin{theorem}
Conjecture~\ref{conj_qs_rigidity} holds for nonrigid Carnot groups.  
\end{theorem}
\noindent
See Theorem~\ref{qs_rigidity_theorem} for a more precise statement.
Note that for rigid groups Conjecture~\ref{conj_qs_rigidity} would follow from the Regularity Conjecture~\ref{conj_regularity_conjecture} and \cite{cowling_ottazzi}.

We now give some indication of the proof of Theorem~\ref{thm_nonrigid_intro}.  

Let $G$ be a nonrigid Carnot group with graded Lie algebra $\fg$.  By \cite[Theorem 1]{ottazzi_warhurst} (see also \cite{doubrov_radko}), the first layer $V_1^\C$ of the complexification $\fg^\C$ contains a nonzero element $Z$ such that $[Z,\fg^\C]$ is a subspace of dimension at most $1$. The first phase of the proof of Theorem~\ref{thm_nonrigid_intro}, which is implemented in Section~\ref{sec_structure_nonrigid_irreducible_first_layer}, is to work out the  algebraic implications of this, which leads to the following trichotomy\footnote{The trichotomy is a  variation on unpublished work of the third author \cite{Xie_quasiconformal_on_non_rigid}.}:
\ben[label={(\alph*)}]
\item $G$ is isomorphic  to $\R^n$, $\H_n$, or $\H_n^\C$ for some $n\geq 1$.
\item There is a linear subspace $\{0\}\subsetneq W\subsetneq V_1$ which is invariant under graded automorphisms of $\fg$, such that $[W,V_j]=\{0\}$ for all $j\geq 2$.
\item $G$ is a {\bf product quotient}: it is of the form $G=\tilde G/\exp(K)$ where $\tilde G=\prod_{j=1}^n\tilde G_j$ and the $\tilde G_j$s are copies of $\H_m$ or $\H_m^\C$ for some $m\geq 1$, and $K\subset\tilde\fg$ is a linear subspace of the second layer $\tilde V_2\subset \tilde \fg$ which satisfies several conditions, see Definition~\ref{def_product_quotient}.  
\een
In case (a) assertion (1) of Theorem~\ref{thm_nonrigid_intro} holds, so we are done.  In case (b), letting $\fh\subset \fg$ be the Lie subalgebra generated by $W$, and $H:=\exp(\fh)$, the $\aut(\fg)$-invariance of $W$ implies that the Pansu differential of $f$ respects the coset foliation of $H$, and then (2) follows by an integration argument.  In case (c), the main obstacle is that there are automorphisms $\Phi:\tilde G/\exp(K)=G\ra G=\tilde G/\exp(K)$ induced by an automorphism $\tilde \Phi:\tilde G\ra  \tilde G$ which permutes the factors of $\tilde G=\prod_j\tilde G_j$, and one has to show that the Pansu differentials  of $f:B(x,r)\ra G$ at different points in $B(x,r)$ induce the same permutation of the factors.  The easiest case is when $K=\{0\}$, i.e. $G=\tilde G$ is a product, and then the constancy of the permutation was addressed in the product rigidity theorem in   \cite{KMX1}.  This follows easily by applying the Pullback Theorem from \cite{KMX1} (See Subsection~\ref{subsec_pullback_theorem})  to the volume forms coming from the factors of $G$.  The treatment of  case (c) in general breaks down into three main subcases depending on the dimension of $K$ and the type of the factors $\tilde G_j$.  The arguments are lengthy and hard to summarize briefly; we refer the reader to Section~\ref{sec_rigidity_product_quotients} and to the beginning of Sections~\ref{sec_dim_k_equals_1}-\ref{sec_higher_product_quotients} for more explanation; see also below for suggestions on how to read this paper.

\bigskip
\subsection*{ Suggestions for reading}~

We suggest reading the paper in the following order:
\bit
\item The Pullback Theorem, Subsection~\ref{subsec_pullback_theorem}.
\item Subsection~\ref{subsec_statement_results_initial_reductions}: precise statement of main results, and reduction to results proven later in the paper.
\item The statements of the Theorem~\ref{thm_irreducible_nonrigid_structure} and Lemma~\ref{lem_converse_product_quotient}, and the definition of product quotient, Definition~\ref{def_product_quotient}.
\item Examples~\ref{ex_diagonal_product_quotient}-\ref{ex_decomposable_product_quotients}.
\item The proof of product rigidity from \cite[Subsection 7.1]{KMX1}.
\item Section~\ref{sec_dim_k_equals_1} on product quotients with $\dim K=1$.
\item Section~\ref{sec_higher_product_quotients} on product quotients of complex Heisenberg groups and higher Heisenberg groups.  This could also be read after Section~\ref{sec_conformal_case}, but the proof is relatively simple and independent of Sections~\ref{sec_dim_k_equals_1} and \ref{sec_conformal_case}. 
\item Lemma~\ref{lem_decomposition_conformally_compact} giving a canonical decomposition of a product quotient in conformal product quotients.
\item Section~\ref{sec_conformal_case} on conformal product quotients with $\dim K\geq 2$.
\eit

\bigskip
\section{Preliminaries}
\label{sec_prelim}~

We refer to \cite{KMX1} for most of the preliminaries, which are common to both papers; this includes background as well as notation and conventions.  Apart from the pullback theorem, the results collected here are specific to the needs of this paper.

\bigskip
\subsection{The Pullback Theorem}
\label{subsec_pullback_theorem}
We give a concise presentation here, and refer the reader to \cite{KMX1} for more details.

Let $G$  be a step $s$ Carnot group with Lie algebra $\fg$, grading $\fg=\oplus_{j=1}^sV_j$, dilation group $\{\de_r:G\ra G\}_{r\in (0,\infty)}$, dimension $N$, and homogeneous dimension $\nu=\sum_jj\dim V_j$.   We let $\{X_j\}_{1\leq j\leq N}$ be a graded basis for $\fg$, and $\{\th_j\}_{1\leq j\leq N}$ be the dual basis. 

The {\em weight} of a subset $I\subset \{1,\ldots,N\}$ is given by 
$$
\wt I:=-\sum_{i\in I}\deg i
$$
where $\deg:\{1,\ldots,N\}\ra \{1,\ldots, s\}$ is defined by $\deg i=j$ iff $X_i\in V_j$.  For a non-zero  left-invariant form
$\alpha = \sum_{I} a_I \theta_I$ we define $\wt(\alpha) = \max \{ \wt I : a_I \ne 0\}$ and set $\wt(0):=-\infty$; here $\th_I$ denotes the wedge product $\La_{i\in I}\th_i$.

We now let $G'$ be a second Carnot group, use primes for the associated structure.

Fix $p>\nu$, and let $f:U\ra G'$ be a $W^{1,p}_{\loc}$-mapping for some open subset $U\subset G$.  If $\om\in\Om^k(G')$ is a differential $k$-form with continuous coefficients, the {\em Pansu pullback of $\om$} is the $k$-form with measurable coefficients $f_P^*\om$ given by 
$$
f_P^*\om(x):=(D_Pf(x))^*\om(f(x))\,,
$$ 
where  $D_Pf(x):\fg\ra \fg'$  is the Pansu differential of $f$ at $x\in U$.

 We will use the following special case of  \cite[Theorem 4.2]{KMX1}:
\begin{theorem}[Pullback Theorem (special case)]  \label{co:pull_back2}
Suppose $\varphi \in C_c^\infty(U)$ and  that $\alpha$ and $\beta$ are closed left invariant forms
which satisfy 
\begin{equation} \label{eq:weight_special}
\deg \alpha + \deg \beta = N -1 \quad \hbox{and} \quad  \wt(\alpha) + \wt(\beta) \le -\nu + 1.
\end{equation}
Then 
\begin{equation} \label{eq:pull_back_identity_special}
 \int_U f_P^*(\alpha) \wedge d(\varphi \beta) = 0.
\end{equation}
\end{theorem}

\bigskip

\subsection{Exterior algebra and interior products}  \label{se:interior_products}
Here we recall some basic facts about interior products. This will be useful for constructing closed left invariant forms of low 
codegree that will be useful in the application of the Pullback Theorem.  In this subsection we work purely in the setting of multilinear algebra. Exterior differentiation of forms will be considered in the next 
subsection.
 
Given a finite   dimensional  vector space $V$ over $\R$ we denote the space of $m$-vectors by $\Lambda_m V$ and the space of alternating
$m$-linear forms by $\Lambda^m V$. The elements of $\Lambda^m V$ can be identified with linear functionals
on $\Lambda_m V$ (or $m$-covectors).

The interior product of a vector $X \in V$ and a $p$-form $\alpha \in \Lambda^p V$ is the $(p-1)$-form
$i_X \alpha$ given by 
\begin{equation} \label{eq:interior_product_wiki}  i_X \alpha(X_2, \ldots, X_p) = \alpha(X, X_2, \ldots, X_p).
\end{equation}
Using the dual pairing of vectors and covectors this can be rewritten as
$$ i_X \alpha(Z) = \alpha(X \wedge Z).$$
For $q \le p$, $X \in \Lambda_q V$ and $\alpha \in \Lambda^p V$ one defines
 $i_X \alpha \in \Lambda^{p-q} V$
by
\begin{equation} \label{eq:interior_product_higher} 
(i_X \alpha)(Z) = \alpha(X \wedge Z)  \quad \forall Z \in \Lambda_{p-q} V.
\end{equation}
Then for $X \in \Lambda_q$ and $Y \in \Lambda_{q'}$
$$ (i_X i_Y \alpha)(Z) = i_Y \alpha(X \wedge Z) = \alpha(Y \wedge X \wedge Z)$$
and thus  
\begin{equation}  \label{eq:i_circ_i}
 i_X \circ i_Y = i_{Y \wedge X}  \quad \forall X \in \Lambda_q, \, \, Y \in \Lambda_{q'}.
 \end{equation}
The interior  product satisfies the graded Leibniz rule
\begin{equation} \label{eq:Leibniz_interior}
 i_X (\beta \wedge \gamma) = i_X \beta \wedge \gamma + (-1)^{\deg \beta} \beta \wedge i_X \gamma 
 \quad \forall X \in V.
 \end{equation}

If $\omega$ is a top-degree form on $V$, then 
\begin{equation} \label{eq:inner_product_volume_form}
\alpha \wedge i_X \omega = i_X \alpha \wedge \omega = \alpha(X) \omega\,,\quad\forall \alpha  \in \Lambda^p V\,,\quad X \in \Lambda_p V.
\end{equation}
This can easily be checked using the standard basis and dual basis and verifying the assertion
for $X$ and $\beta$ being simple vectors in terms of these basis vectors.

We finally briefly discuss the interior product on quotient spaces. Let $\tilde V$ be a finite-dimensional vector space
over $\R$, let $K \subset \tilde V$ be a subspace and let $V = \tilde V/ K$. 
We say that $\beta \in \Lambda^p \tilde V$ annihilates $K$ if $\beta(X_1, \ldots, X_p) = 0$ whenever at least
one of the $X_i$ satisfies $X_i \in K$. If $\tilde \alpha \in \Lambda^p \tilde V$ annihilates $K$ then
it gives rise to a unique  form $\alpha \in \Lambda^p V$ defined by
\begin{equation}  \label{eq:lift_form_annihilate_K}
 \alpha(X_1 + K, \ldots, X_p + K) = \tilde \alpha(X_1, \ldots, X_p)  \quad \forall X_i \in \tilde V.
 \end{equation}
Conversely for every $\alpha \in \Lambda^p V$ there exists a unique lift $\tilde \alpha \in \Lambda^p \tilde V$ 
which annihilates $K$ and satisfies \eqref{eq:lift_form_annihilate_K}. Consider such $\alpha$ and $\tilde \alpha$
as well as $X \in \tilde V$. Then 
$$\widetilde{ i_{X + K} \alpha }= i_X \tilde \alpha.$$
To simplify the notation we sometimes use the notation $i_X \alpha$ for the form  $i_{X + K} \alpha$
and identify
$i_X \alpha$ and $i_X \tilde \alpha$.

 \subsection{Exterior derivatives of left invariant forms with low codegree}
 \label{se:exterior_derivative_volume}

In this subsection we recall some useful formulae for exterior derivatives of forms on unimodular Lie groups, particularly for low codegree forms defined using interior products of vectors with the volume form.  Recall that  Carnot groups are unimodular \cite[Theorem 1.2.10]{corwin_greenleaf_nilpotent_lie_groups}; we will only use the results in this section in the case of Carnot groups, but for the sake of logical clarity we have formulated them for unimodular groups, because they hold in such generality.

Let $G$ be a Lie group. 
We identify the corresponding  Lie algebra $\fg$ with the space of left invariant vector  fields on $G$, i.e., with the left
invariant sections in the tangent bundle   $TG$. 
Similarly we identify the space of $1$-forms $\Lambda^1 \fg$
with the space of    left invariant 
    sections of the cotangent bundle $T^*G$ and the spaces of $p$-vectors in $\fg$ and $p$-forms
on $\fg$ with the spaces of  
left invariant sections of the corresponding tensor bundles.

For a vector   field $X$ let $\mathcal L_X$ denote the Lie derivative with respect to $X$. 
We will use Cartan's formula
\begin{equation} \label{eq:cartan_magic}
\mathcal L_X \alpha = d(i_X \alpha) + i_X d \alpha
\end{equation}
as well as the relation
\begin{equation} \label{eq:commutator_Lie_derivative}
i_{[X,Y]} = [\mathcal L_X, i_Y], 
\end{equation}
see, e.g., \cite[Theorem 9.9]{Michor}.

\begin{lemma}
\label{lem_volume_form_contraction_exterior_derivative} 
Let $G$ be a unimodular  group and let $\omega$ be a bi-invariant  volume form on $G$. Then 
\begin{eqnarray}  \label{eq:exterior_codegree1}
d(i_X \omega) &=& \mathcal L_X \omega = 0 \quad \forall X \in \fg, \\
  \label{eq:exterior_codegree2}
d(i_X i_Y  \omega) &= & i_{[X,Y]}  \omega   \quad \forall X, Y \in \fg, \\
 \label{eq:exterior_codegree3}
d(i_X i_Y i_Z \omega) &=& i_A  \omega   \quad \forall X, Y, Z \in \fg, \quad
\hbox{where} \\
 \label{eq:exterior_codegree3bis}
A &=& 
X\wedge [Y, Z]+Y\wedge[Z,X]+Z\wedge[X,Y] .
\end{eqnarray}
Finally if
\begin{equation}  \label{eq:exterior_codegree5_assumption}
X, X', X'', Z \in \fg \quad \hbox{and} \quad [X,Z] = [X', Z] = [X'', Z] = 0
\end{equation}
then
\begin{equation} 
  \label{eq:exterior_codegree5}
d(i_Z  i_{X''} i_{X'} i_X  \omega) = -  i_Z d( i_{X''} i_{X'} i_X  \omega)
\end{equation}
\end{lemma}

\begin{proof} 
Since $\omega$ has top degree it is closed and thus Cartan's formula gives
\begin{equation}  \label{eq:Lie_codegree1}
d(i_X \omega) = \mathcal L_X \omega
 \quad \forall X \in \fg.
\end{equation}
Since the flow of a left invariant vector field is given by right translation and since $\omega$ is also 
        right                  invariant     
we have $\mathcal L_X \omega = 0$.

To show \eqref{eq:exterior_codegree2} we use Cartan's formula  and 
 \eqref{eq:exterior_codegree1} to get
\begin{eqnarray*}
d(i_X i_Y \omega) &= & 
  \mathcal L_X (i_Y \omega) - i_X d (i_Y \omega) =   \mathcal L_X (i_Y \omega) \\ 
  &=&    \mathcal L_X (i_Y \omega)  - i_Y (\mathcal L_X \omega)  
  \underset{\eqref{eq:commutator_Lie_derivative}}{=}
   i_{[X,Y]} \omega.
\end{eqnarray*}

To prove \eqref{eq:exterior_codegree3} we start again from Cartan's formula
\begin{eqnarray*}
d(i_X i_Y i_Z \omega) &=& \mathcal L_X (i_Y i_Z \omega) - i_X d(i_Y i_Z \omega) \\
&=& [\mathcal L_X, i_Y] i_Z \omega + i_Y \mathcal L_X i_Z \omega - i_X i_{[Y,Z]} \omega \\
&=& i_{[X,Y]} i_Z \omega + i_Y [\mathcal L_X, i_Z] \omega + i_Y i_Z \underbrace{\mathcal L_X \omega}_{=0}
- i_X i_{[Y,Z]} \omega.
\end{eqnarray*}
Now the formula for $A$ follows from \eqref{eq:commutator_Lie_derivative} and  \eqref{eq:i_circ_i}.

To prove   \eqref{eq:exterior_codegree5} we use that $Z$ commutes with  each of $X''$, $X'$, $X$.     
   Thus successive application of  \eqref{eq:commutator_Lie_derivative} and
  \eqref{eq:exterior_codegree1}
give
$$ \mathcal L_Z (i_{X''} i_{X'} i_X \omega) = i_{X''} \mathcal L_Z (i_{X'} i_X \omega) = \ldots =
i_{X''} i_{X'} i_X \mathcal L_Z \omega = 0.
$$
Now  \eqref{eq:exterior_codegree5} follows from Cartan's formula  \eqref{eq:cartan_magic}.
\end{proof}

The exterior derivative of a left invariant form is left invariant since left translation is smooth
and hence commutes with exterior differentiation. Thus the exterior derivative induces a linear map on $\Lambda^* \fg$
via
\begin{equation}  
 d \alpha := (d \tilde \alpha)(e).
 \end{equation}
where $\tilde \alpha$ is the left invariant extension of $\alpha$. 
\begin{lemma}  \label{le:d_acts_on_algebra} Let $G$ be a Lie group and let $\alpha \in \Lambda^k \fg$. Then
\begin{equation}  \label{eq:exterior_derivative_on_algebra}
d\alpha(X_0, \ldots, X_k) = \sum_{0 \le i < j \le k }  (-1)^{i+j} \alpha([X_i, X_j], X_0, \ldots, \widehat X_i, \ldots, \widehat X_j, \ldots, X_k).
\end{equation}
\end{lemma}
\begin{proof} See \cite[Lemma 14.14]{Michor}.
\end{proof}
For one forms this reduces to 
\begin{equation}  \label{eq:exterior_derivative_on_algebra_one_form}
d\alpha(X_0, X_1) = - \alpha([X_0, X_1]).
\end{equation}

Let $\Phi \in \aut(\fg)$. We define the pullback     of $\alpha \in \Lambda^k \fg$ by
$(\Phi^* \alpha)(X_1, \ldots, X_k) := \omega(\Phi(X_1), \ldots, \Phi(X_k))$ for $X_i \in \fg$. 
Then  \eqref{eq:exterior_derivative_on_algebra}  implies that
 \begin{equation} \label{eq:d_commutes_with_aut(fg)}
d \Phi^*\omega = \Phi^* d \omega.
\end{equation}
Note that if $\Phi$ is the tangent map of group homomorphism, then the above definition of pullback agrees with the usual one, i.e. one first extends $\al\in \Lambda^k\fg$ to a left invariant form and applies the usual pullback operation.

\bigskip

\subsection{Miscellaneous}

The following Lemma is the distributional version of the statement: ``$g$ constant along a line
implies $h \circ g$ constant on this line''. 

We recall that  the flow $\tau_s$ of a left invariant vector field $X$ is given by right translation by the $1$-parameter group generated by $X$, i.e.  $\tau_s(p)=p\exp(sX)$  see \cite{Lee}, Thm. 5.74 or \cite[4.17]{Michor}.  We also recall that Carnot groups are unimodular, since they are nilpotent.   See, for example, \cite[Theorem 1.2.10]{corwin_greenleaf_nilpotent_lie_groups}.

\begin{lemma}  \label{le:compact_directional_constancy} 
Let $G$ be a unimodular Lie group with bi-invariant Haar measure $\mu$,   $U \subset G$ be an open subset,  $g \in L^1_{loc}(U; \R^l)$.  For every  $X \in \fg$, if   $\tau_s:=r_{\exp sX}$ denotes $1$-parameter group of    right
translations generated by $X$, then:  

\begin{enumerate}
\item We have 
\begin{equation}  \label{eq:directional_constant}  \int_{U}   g  \,  X \varphi  \, d\mu  = 0 
\quad \forall \varphi \in C_c^\infty(U)
\end{equation}
if and only if for every compact set $V \subset U$ there exists an $\eps > 0$ such that
\begin{equation}  \label{eq:directional_constant_bis} 
\hbox{$g \circ \tau_s = g$ \hbox{ a.e.\  on $V$}  \quad $\forall s \in (-\eps, \eps)$.}
\end{equation}
\item Let $h: \R^l \to \R$ be a Borel function with $|h(t)| \le C_1 |t| + C_2$. 
Then  \eqref{eq:directional_constant} implies that
\begin{equation}  \label{eq:directional_constant_ter} 
   \int_{U}  (h \circ  g)   \,  X \varphi  \, d\mu  = 0 \quad \forall \varphi \in C_c^\infty(U).
   \end{equation}
\end{enumerate}
\end{lemma}   

\begin{proof} 
The main point is that \eqref{eq:directional_constant} is equivalent to
\begin{equation} \label{eq:distributional_derivative_integrated}
\int_{U}  g \, ( \varphi \circ \tau_s - \varphi)  \, d\mu  = 0 \quad \forall  \varphi \in C_c^\infty(U),
s \in (-\eps_\varphi, \eps_\varphi)
\end{equation}
where $\eps_\varphi$ depends on the support of $\varphi$, $U$ and $X$. 
 Since $\mu$ is bi-invariant \eqref{eq:distributional_derivative_integrated} is equivalent to
$$ \int_{U}  (g \circ \tau_{-s} - g  ) \,  \varphi \, d\mu  = 0 \quad \forall  \varphi \in C_c^\infty(U),
s \in (-\eps_\varphi, \eps_\varphi)$$
or to $g \circ \tau_s = g$ a.e. This concludes the proof of the first assertion. 

To prove the second assertion note that  the condition $h$ is  Borel ensures  that $h \circ g$ is measurable.
The growth condition implies that $h \circ g \in L^1_{loc}(U)$. Now apply $h$ to both sides of 
\eqref{eq:directional_constant_bis} and use the equivalence of \eqref{eq:directional_constant_bis} 
and \eqref{eq:directional_constant}. 
\end{proof}

\bigskip\bigskip
Now let $G$, $G'$ be Carnot groups, where $G$ has homogeneous dimension $\nu$.  Let  $\fh\subset\fg$, $\fh'\subset\fg'$ be Lie subalgebras, where $\fh$ is horizontally generated (i.e. $\fh$ is generated by $\fh\cap V_1$), and let $H:=\exp(\fh)$, $H':=\exp(\fh')$.  Note that $H$, $H'$ are closed subgroups since the exponential maps are diffeomorphisms.

\begin{lemma}
\label{lem_preservation_cosets}
There exists $\La\in [1,\infty)$ such that if $x\in G$, $r\in (0,\infty)$, $f:B(x,\La r)\ra G'$ is a $W^{1,p}_{\loc}$-mapping for some $p>\nu$, and the Pansu differential satisfies $D_Pf(x)(\fh)\subset\fh'$, then for every $y\in G$ the image of $yH\cap B(x,r)$ under $f$ is contained in a single left coset of $H'$.
\end{lemma}
\begin{proof}
This is a slight variation on \cite[Proposition 3.4]{Xie_Pacific2013}.

Since $\fh$ is horizontally generated, it is invariant under the dilation $\de_r:G\ra G$ for every $r\in (0,\infty)$.  Therefore there is a direct sum decomposition $\fh=\oplus_jV_j^{\fh}$ where $V_j^{\fh}:=\fh\cap V_j$, and  $(H,\fh=\oplus_jV_j^{\fh})$ is a Carnot group.  We may equip the first layer $V_1^{\fh}$ with the restriction of the inner product on $V_1$, and let $d_{CC}^H$ be the corresponding Carnot-Caratheodory metric on $H$.  

Note that there is a constant $C_1$ such that 
$$
d_{CC}(x,x')\leq d_{CC}^H(x,x')\leq C_1 d_{CC}(x,x')\,;
$$ 
the first inequality follows directly from the definitions, while the second follows for pairs $(x,x')=(e,y)$ with $d_{CC}(e,y)=1$ by compactness, and for general pairs by applying translation and scaling.

Let $\La:=2C_1\La_0+1$, where $\La_0$ is the constant given by Lemma~\ref{lem_quantitative_accessibility}.  Suppose  $x$, $r$, and $f$ are as in the statement of Lemma~\ref{lem_preservation_cosets}; by rescaling and left translation we may assume that $x=e$ and $r=1$.  Arguing as in \cite[Proposition 3.4]{Xie_Pacific2013}, for every path $ \ga:  [0,1]\ra B(e,\La)$ with constant horizontal velocity  and with  image   lying  in  a coset of $H$,    the image $f(\ga([0,1]))$ is contained in a single coset of $H'\subset G'$.  

Choose $y\in G$, and $x,x'\in yH\cap B(e,1)$.  Then $d_{CC}(x,x')<2$, so $d_{CC}^H(x,x')<2C_1$.  Therefore  by Lemma~\ref{lem_quantitative_accessibility}, for every $\eps>0$ there is a piecewise horizontal path $\ga:[0,1]\ra yH\cap B(e,\La)$     with  $\ga(0)=x$ and $\ga(1)\in B(x',\eps)$; since  $f(\ga([0,1]))$ lies in a single coset of $H'$ and $\eps$ is arbitrary, by continuity it follows that $x$ and $x'$ belong to the same coset of $H'$.    
\end{proof}

\bigskip\bigskip

\begin{lemma}
\label{lem_quantitative_accessibility}
Let $H$ be a Carnot group.  Then there exists $\La_0>1$ with the following property.

Let $\Gamma$ be the collection of paths $\ga:[0,1]\ra H$ such that $\ga(0)=e$ 
and $\ga'$ is piecewise constant and horizontal, i.e. there exists a partition $0=t_0<\ldots<t_n=1$, and for every  $1\leq j\leq n$  elements $X_j\in V_1$, $h_j\in H$,  such that 
$$
\ga(t)=h_j\exp((t-t_{j-1})X_j)\quad\text{for all}\quad t\in [t_{j-1},t_j].
$$  

Then the set 
$$
\{\ga(1)\mid \ga\in \Gamma\,,\;\ga([0,1])\subset B(e,\La_0)\}\cap B(e,1)
$$ 
is dense in $B(e,1)$.
\end{lemma}
\begin{proof}
Let $\hat H$ be the closure of the subgroup generated by the collection of horizontal $1$-parameter subgroups.  Then $\hat H$ is a Lie subgroup, and its Lie algebra $\hat \fh$ contains $V_1$, so $\hat H=H$.  It follows that the collection $\{\ga(1)\mid \ga\in\Gamma\}$ is dense in $H$.  

Now choose a finite collection $\Gamma_1=\{\ga_1,\ldots,\ga_n\}\subset \Gamma$ such that the endpoints are $\frac12$-dense in $B(e,1)$, i.e. for every $x\in B(e,1)$ we have $d(x,\{\ga_j(1)\}_{1\leq j\leq n}\})<\frac12$.  Then for some $\La_1$  and every $1\leq j\leq n$ the image $\ga_j([0,1])$ is contained in the ball $B(e,\La_1)$. 

Pick $y\in B(e,1)$.  By induction, there is a sequence $y_0=e, y_1,\ldots$ and for every $j\geq 1$ an element $i_j\in\{1,\ldots, n\}$ such that $d(y_j,y)< 2^{-j}$ for all $j\geq 0$, and the path $\ell_{y_j}\circ\de_{2^{-j}}\circ\ga_{i_j}$ joins $y_j$ to $y_{j+1}$.  By concatenating, for every $j$ we obtain a path $\ga\in \Ga$ where $\ga(1)=y_j$ and $\ga([0,1])\subset B(e,\La_0)$, where $\La_0:=2\La_1$.
\end{proof}

\bigskip

\section{Quasisymmetric rigidity of nonrigid Carnot groups}
\label{sec_qs_rigidity_carnot_groups}~

In this section we state our main  rigidity results,   and prove them using results that will be established later in the paper. We also prove a localized version of a result of Le Donne-Xie which deduces quasisymmetric rigidity (i.e. automatic improvement to bilipschitz regularity) for mappings that preserve certain types of foliations.

\subsection{Statement of results and some initial reductions}
\label{subsec_statement_results_initial_reductions}

\begin{theorem}[Quasisymmetric rigidity]
\label{qs_rigidity_theorem}
Let $G$ be a Carnot group that is nonrigid in the sense of Ottazzi-Warhurst, and assume that $G$ is not isomorphic to $\R^n$ or a real Heisenberg group.  Let  $f:G\supset U\ra U'\subset G$ be an $\eta$-quasisymmetric homeomorphism between open subsets, where $U=B(x,r)$ for some $x\in G$, $r\in (0,\infty)$.
Then, modulo post-composition with a Carnot rescaling, the restriction of $f$ to the subball $B(x,\frac{r}{K})$  is a $K$-bilipschitz homeomorphism onto its image, where $K=K(\eta)$.  

Consequently, modulo post-composition with a Carnot rescaling
every locally $\eta$-quasisymmetric homeomorphism between open subsets of $G$ is locally $K=K(\eta)$-bilipschitz, and any global $\eta$-quasisymmetric homeomorphism onto its image $G\ra U'\subset G$ is a (surjective) $K=K(\eta)$-bilipschitz homeomorphism $G\ra G$.
\end{theorem}

\begin{remark}
If desired, one may replace $\frac{r}{K}$ with $\frac{r}{2}$ in the statement of Theorem~\ref{qs_rigidity_theorem} at the cost of increasing $K$; this follows readily by combining $\eta$-quasisymmetry with the bilipschitz control on small balls.
\end{remark}

We will show that apart from the complex Heisenberg case, quasisymmetric homeomorphisms are ``reducible'' in the following sense.

\begin{definition}
\label{def_preserving_coset_foliation}
Let $H$ be a closed subgroup of a Lie group $G$, and $f:G\supset U\ra U'\subset G$ be a homeomorphism between open subsets.  Then {\bf $f$ preserves the coset foliation of $H$} if for every $g\in G$ there is a $g'\in G$ such that $f(U\cap gH)=U'\cap g'H$.
\end{definition}

\begin{proof}[Proof of Theorem~\ref{qs_rigidity_theorem} using Theorem~\ref{thm_nonrigid_intro}]
By \cite[Theorem 2.12]{KMX1} we know that $f$ is a $W^{1,p}_{\loc}$-mapping for some $p>\nu$. Also, since $B(x,r)$ is connected, the sign of the determinant of Pansu differential $D_Pf(x):\fg\ra \fg$ is constant almost everywhere.  Theorem~\ref{qs_rigidity_theorem} now follows from \cite[Corollary 8.3]{KMX1}   when $G$ is a complex Heisenberg group; otherwise it follows from Theorem~\ref{thm_nonrigid_intro}  and Proposition~\ref{lem_subalgebra_ideal} below.  

\end{proof}

\bigskip

\begin{proof}[Proof of Theorem~\ref{thm_nonrigid_intro}]
Let  $f: G\supset B(x,r)\ra G$ be a $W^{1,p}$-mapping   as in the statement of the theorem.

{\em Case 1. There is  an $\aut(\fg)$-invariant linear subspace $\{0\}\neq W \subsetneq V_1$ such that $[W,V_i]=\{0\}$ for $i\geq 2$.}  Let $\fh$ be the subalgebra generated by $W$, and $H$ the (closed) subgroup with Lie algebra $\fh$.  Then the Pansu differentials of $f$ preserve the tangent space of the coset foliation of $H$, so by Lemma~\ref{lem_preservation_cosets}  the restriction of $f$ to the subball $B(x,\frac{r}{K})$ preserves the cosets of $H$, for $K=K(G)$.   So we are done in this case. 

{\em Case 2. There is no subspace $W\subset V_1$ as in Case 1.}  In this situation, by using the fact that $G$ is a nonrigid group and analyzing the algebraic implications, we are able to deduce that $G$ has a very special structure.   Specifically, we may  apply Theorem~\ref{thm_irreducible_nonrigid_structure}, so conclusions (1)-(5) hold; moreover, since $G$ is not a complex Heisenberg group, we have $n\geq 2$.  This implies that $G$ is a product quotient $\tilde G/\exp(K)$ where $\tilde G$ has at least two factors, see Definition~\ref{def_product_quotient}, Lemma~\ref{lem_converse_product_quotient}, and Remark~\ref{rem_product_quotient_same_as_classification}.   Let  $W=V_{1,i}$ for some $1\leq i\leq n$, and let $H$ be the subgroup generated by $W$;   note that since $G$ is a step $2$ Carnot group, the Lie algebra $\fh$ of $H$ satisfies $[\fh,V_i]=\{0\}$ for all $i\geq 2$.  We may apply Theorem~\ref{thm_rigidity_product_quotient} and Lemma~\ref{lem_preservation_cosets}  to obtain that the restriction of $f$ to the subball $B(p,\frac{r}{K})$ preserves the cosets of $H$, after possibly post-composing with an automorphism which permutes the subalgebras $\fg_1,\ldots,\fg_n$; this theorem is proved by means of a subtle analysis based on the pullback theorem (Sections~\ref{sec_rigidity_product_quotients}-\ref{sec_conformal_case}). Here we are using the fact that the permutation of the subalgebras $\fg_i$ induced by the Pansu differential is locally constant (by Theorem~\ref{thm_rigidity_product_quotient}), and therefore and therefore constant since the ball $B(p,r)$ is connected.
\end{proof}

\bigskip
\subsection{Promoting locally quasisymmetric homeomorphisms to locally bilipschitz homeomorphisms}

\mbox{}
We recall that it was shown in \cite{LeDonne_Xie} (see also \cite{Shan_Xie, Xie_Pacific2013}) that quasisymmetric homeomorphisms are necessarily bilipschitz provided they preserve certain types of foliations.  The purpose of this subsection is to prove a localized version of this assertion in certain cases, which include the ones arising in this paper.  These results may be viewed as refinements of the elementary observation that if $f_1:\R^k\ra \R^k$, $f_2:\R^l\ra\R^l$ are quasisymmetric homeomorphisms, then the  product 
$$
f=(f_1,f_2):\R^k\times\R^l\lra \R^k\times \R^l
$$
is quasisymmetric only if $f_1$ and $f_2$ (and hence also $f$) are bilipschitz.  The rough idea of \cite{Shan_Xie, Xie_Pacific2013,LeDonne_Xie}, in the simplest case (as above) where the leaves of the foliation are parallel,  is that for any quasisymmetric homeomorphism preserving the foliation,  the expansion/compression transverse to the foliation is comparable to the expansion/compression along the leaves.   

\bigskip

\begin{proposition}
\label{lem_subalgebra_ideal}

 Let $G$ be a Carnot group with graded Lie algebra $\fg = \oplus_{i=1}^r V_i$, and let $W \subset V_1$ be a subspace of the first stratum. 
Assume that $\{0\} \ne W \subsetneq V_1$ and 
\begin{equation} \label{eq:W_commutes}
[W, V_i] = \{0\} \quad \forall i \ge 2.
\end{equation}
Let $\fh$ be the subalgebra generated by $W$ and let $H$ be the subgroup  with Lie algebra   $\fh$.  Then there is an $L=L(\eta)$ such that any $\eta$-quasisymmetric homeomorphism 
$$
f: G\supset B(x,r)\ra U\subset G
$$  
which respects the coset foliation of $H$ (see Definition~\ref{def_preserving_coset_foliation}) is, modulo postcomposition with a suitable Carnot rescaling, $L$-bilipschitz on the subball $B(x,L^{-1}r)$.  
\end{proposition}

The following result and its proof were  inspired by Lemma 4.4 in \cite{LeDonne_Xie} 
 where it is shown that cosets are either parallel or diverge at $\infty$. 
\begin{proposition} \label{pr:one_sided_almost_parallel}
Let $\fh$, $\fg$, $H$ and $G$  be as in Proposition~\ref{lem_subalgebra_ideal}. Then there exists a constant $c >0$
with the following property. If
\begin{equation}
x_0, y \in G, \quad x_0^{-1} y \notin H  \quad  \hbox{and  $X \in W$}
\end{equation}
then at least one of the following holds: 
\begin{equation} 
\label{eq:almost_parallel_plus}
d(x_0 \exp(tX), yH) \ge c d(x_0, yH) \quad \forall t \ge 0
\end{equation}
or 
\begin{equation} 
d(x_0 \exp(tX), yH) \ge c d(x_0, yH) \quad \forall t \le 0.
\end{equation}
\end{proposition}

Replacing $y$ by $x_0^{-1} y$ we see that it suffices to prove the proposition for $x_0 = e$. 
The heart of the matter is to estimate the distance from $\exp(tX)^{-1} y^{-1} \exp(tX)$ to $H$ from below. 
The key calculation is contained in the following proposition.
Recall   that the exponential map $\exp:\fg\ra G$ is a diffeomorphism, so we may pull back the group operation to the algebra.  We denote  by $\ast$ the resulting operation, i.e.
$\exp(A * B) = \exp A \exp B$. For $A \in \fg$ we denote by $A_i$ the projection to the $i$-th stratum.

\begin{lemma}
Let $\fh$ be as in Proposition~\ref{pr:one_sided_almost_parallel}.
Define
$$N(A,B) := A \ast B - (A + B).$$
Then 
\begin{equation} \label{eq:almost_parallel_commute_h2}
[[W, W], \fg] = \{0\}
\quad \hbox{and hence} \quad \fh = W \oplus [W,W].
\end{equation}
Moreover, 
\begin{equation} \label{eq:almost_parallel_induct_commute}
[X, \oplus_{i \ge 2} V_i] = \{0\}  \quad \Longrightarrow \quad [ [X, \fg], \oplus_{i \ge 2} V_i] = \{0\}, 
\end{equation}
and for all $A, A' \in \fh$ and $B \in \fg$
\begin{eqnarray} \label{eq:conjugation_compression}
(-A) \ast B \ast A &=& B - [A_1, B_1],   \\ 
 \label{eq:bch_compression}
N(A',B) &=& N(A'_1, B_1),\\ 
\label{eq:full_bch_compression}
 A' * (- A)  \ast B \ast A &=& A' + B   - [A_1,B_1] + N(A'_1, B_1).
\end{eqnarray}
\end{lemma}

\begin{proof} The  identities in   \eqref{eq:almost_parallel_commute_h2}
and  \eqref{eq:almost_parallel_induct_commute} follow from the Jacobi identity.
  To prove  \eqref{eq:conjugation_compression} recall that 
  $(- A) \ast B \ast A = e^{-\ad A} B$ and use that $(\ad A)^2=0$ since $[A, \oplus_{i \ge 2} V_i] = 0$. 
  The Baker-Campbell-Hausdorff formula implies that $N(A',B)$ is a linear combination of
  iterated commutators of $A'$ and $B$. Since $[A',\oplus_{i \ge 2} V_i] = 0$ only commutators of the 
  form $(\ad B)^j A'$ appear. Now \eqref{eq:almost_parallel_commute_h2} and 
  inductive application of  \eqref{eq:almost_parallel_induct_commute}
  yield $(\ad B)^j A' = (\ad B)^j A'_1 = (\ad B_1)^j A'_1$. This implies \eqref{eq:bch_compression}.
Finally  \eqref{eq:full_bch_compression} follows by applying \eqref{eq:bch_compression} with 
$B$ replaced with $(-A) * B * A$ and using  \eqref{eq:conjugation_compression}.
\end{proof}

Recall that we may define a quasinorm on $\fg$ by
$$
\| X \|:=\sum_j\|X_j\|_0^{\frac{1}{j}}\,,
$$
where $\|\cdot\|_0$ denotes some fixed Euclidean norm on $\fg$, and $X=\sum_jX_j$ is the layer decomposition of $X$. 
The distance $d$ on $G$ is equivalent to the quasinorm, i.e. 
$C^{-1} \| A \| \le d(\exp A, e) \le C \| A\|$, see \cite[2.15 and 2.19]{heinonen_calculus_carnot_groups} or \cite[Proposition 11.15]{hajlasz_koskela_sobolev_met_poincare}.

\begin{proof}[ Proof of Proposition~\ref{pr:one_sided_almost_parallel}]
We may assume $x_0=e$.  Using a good representative of $yH$ and a suitable dilation we may further assume that
$d(e,y) = d(e, yH) = 1$.
Let 
$ x(t) = \exp(tX)$.  
Then
\begin{equation} \label{eq:almost_parallel_dist}
\begin{aligned}
d( x(t), yH) &= d(y^{-1} x(t), H) = d(x(t)^{-1} y^{-1} x(t), H)\\
& = \min_{x' \in H} d( x'^{-1} x(t)^{-1} y^{-1} x(t) , e).
\end{aligned}
\end{equation}
Let 
$ y = \exp Y$ and 
$x' = \exp X'$.
Then \eqref{eq:full_bch_compression} implies that  
\begin{equation}  \label{eq:almost_parallel_XprimeZ}
\begin{aligned}
&\exp^{-1} (x'^{-1}  x(t)^{-1} y^{-1} x(t)) \\=& (-X') \ast (-tX) \ast (-Y) \ast (tX)\\
=&  -Y - X'  +  t[X,Y_1] + N(-X'_1, -Y_1).
\end{aligned}
\end{equation}

Thus, if the assertion is false then there exist 
$$Y_n \in \fg, \quad X_n \in W, \quad s_n > 0, \quad t_n > 0, \quad X_n^\pm \in \fh$$
such that 
\begin{equation} \label{eq:almost_parallel_contra1}
d(e, y_n) = d(e, y_n H)  = 1,  \quad y_n = \exp Y_n
\end{equation} 
and
\begin{eqnarray} \label{eq:almost_parallel_contra2}
\|  -Y_n - X_n^+  +  s_n [X_n,Y_{n,1}] + N(-X^+_{n,1}, -Y_{n,1})\| &\to&  0,\\
 \label{eq:almost_parallel_contra2bis}
\|  -Y_n - X_n^-  -  t_n [X_n,Y_{n,1}] + N(-X^-_{n,1}, -Y_{n,1})\| &\to&  0.
\end{eqnarray} 
Changing the sign of $X_n$ if needed we may assume that 
$ \lambda_n := t_n/s_n\le 1$.
By assumption $d(\exp Y_n, e) = 1$. Hence $\| Y_n\|$ is bounded. 
Passing to  subsequences we may assume that
$$ Y_n \to \bar Y, \quad \lambda_n = \frac{t_n}{s_n} \to \lambda \in [0,1].$$
Using  \eqref{eq:almost_parallel_contra2} 
and  \eqref{eq:almost_parallel_contra2bis}
we get for the first layer quantities 
\begin{eqnarray} \label{eq:almost_parallel_first_layer}
X_{n,1}^\pm &\to& - \bar Y_1  \quad \hbox{and} \quad \bar Y_1 \in \fh.
\end{eqnarray}
Hence 
\begin{equation}  \label{eq:almost_parallel_first_layer_bis}
N(-X^\pm_{n,1}, -Y_{n,1}) \to N(\bar Y_1, - \bar Y_1) = 0.
\end{equation}
Note that $X_n \in V_1$ and thus $[X_{n}, Y_{n,1}] \in V_2$.  Moreover $X^\pm_n  \in \fh \subset V_1 \oplus V_2$. Thus 
 \eqref{eq:almost_parallel_contra2}  
 and \eqref{eq:almost_parallel_first_layer_bis}
 imply
\begin{equation}
\bar Y_i  = 0   \quad \forall i \ge 3. \label{eq:almost_parallel_ith_layer}
\end{equation}
It remains to analyze the behaviour in the second stratum. 
Set 
$ a_n = s_n [X_n, Y_{n,1}]=s_n[X_{n,1},Y_{n,1}]$. Equations 
 \eqref{eq:almost_parallel_contra2},  \eqref{eq:almost_parallel_contra2bis} 
 and \eqref{eq:almost_parallel_first_layer_bis}
and the definition $\lambda_n = t_n/ s_n$ imply that 
\begin{eqnarray}
-Y_{n,2} - X_{n,2}^+ + a_n  \to 0, \quad  
-Y_{n,2} - X_{n,2}^- - \lambda_n  a_n \to 0.
\end{eqnarray}
Let $P^+: V_2 \to V_2$ denote the orthogonal projection onto the orthogonal complement of $\fh\cap V_2 = [W,W]$.
Then we get 
$$ P^+ a_n \to P^+ \bar Y_2,  \quad  \lambda_n P^+ a_n \to - P^+  \bar Y_2.$$
Since $\lambda_n \to \lambda \in [0,1]$ we deduce that $\lambda P^+ \bar Y_2 = - P^+ \bar Y_2$
and hence $P^+ \bar Y_2 = 0$. Together with \eqref{eq:almost_parallel_first_layer}
and \eqref{eq:almost_parallel_ith_layer} it follows that 
$ Y_n \to   \bar Y \in \fh$
and hence $(y_n)^{-1} = \exp (-Y_n) \to \exp (-\bar Y)  \in H$. This contradicts the assumption
$d((y_n)^{-1}, H) = d(e, y_n H) = 1$.
\end{proof}

\bigskip\bigskip
\begin{proof}[Proof of Proposition~\ref{lem_subalgebra_ideal}]
It suffices to consider the case when $x=e$, $r=1$, $f(e)=e$, and the image $f(B(e,1))$ has diameter $1$.   We will show that for $L=L(\eta)$, and  every $x,y\in B(e,L^{-1})$, we have $d(f(x),f(y))\geq L^{-1}d(x,y)$.  The lemma then follows by applying this estimate to the inverse homeomorphism, and adjusting $L$.

To prove the assertion by contradiction, assume that there is a sequence $\{f_j:B(e,1)\ra U\subset G\}$ of $\eta$-quasisymmetric homeomorphisms with $f(e)=e$, $\diam(f_j(B(e,1)))=1$, and there are sequences $\{x_j\},\{y_j\}\subset B(e,j^{-1})$ with 
$d(f_j(x_j),f_j(y_j))\leq j^{-1}d(x_j,y_j)$.  By $\eta$-quasisymmetry and our normalization, it follows that $r_j:=d(x_j,y_j)\ra 0$ and $d(f_j(x_j), e) \to 0$.
Let $K = \{ g \in G :  g^{-1} h g \in H \, \, \forall h \in H\}$ be the normalizer of $H$ in $G$ and let $\rho_j = d(f_j(x_j), f_j(y_j))$.

\begin{claim}
There exist $z_j$ such that
\begin{equation}  \label{eq:exists_good_zj}
d(x_j, z_j H) = r_j, \quad d(f_j(x_j), f_j(z_j) H) \le C_1 \rho_j
\end{equation}
where $C_1 = \eta(1)$ 
and
\begin{equation} \label{eq:images_parallel}
f_j(x_j)^{-1} f_j(z_j) \in K.
\end{equation}
\end{claim}
Note that  \eqref{eq:images_parallel} implies that the cosets $f_j(x_j)H$ and $f_j(z_j) H$ are parallel and in particular
\begin{equation}  \label{eq:images_parallel2}
\begin{aligned}
d(f_j(x_j) h, &f_j(z_j) H) = d(f_j(x_j), f_j(z_j) H) \\
&\le C_1 \rho_j \le   C_1 j^{-1} r_j \quad \forall h \in H.
\end{aligned}\end{equation}
\begin{proof}[Proof of claim]
We first note that $K \ne H$; this holds more generally see, e.g., \cite[Lemma 4.2]{LeDonne_Xie}, however in our case one can see this as follows. By equation (\ref{eq:W_commutes}) from the hypotheses, the normalizer contains $\oplus_{i\geq 2}V_i$,  so we are done unless $\fh$ contains $\oplus_{i\geq 2}V_i$, in which case $\fh$ is an ideal, so its normalizer is all of $\fg$.

To see that such a $z_j$ exists, note that there exists $\zeta \in K$ such that $d(\zeta, H) = d(\zeta, e) \ne 0$. Applying a Carnot dilation we may assume in addition that
$d(\zeta, e) = 1$. For $s> 0$  let $\de_s$ denote the Carnot
   dilation by $s$  and define
$$ g(s) := d(x_j,  f_j^{-1}( f_j(x_j) \de_s \zeta) H).$$
If there exists an $s_j \in [0, C_1 \rho_j]$ such that   $g(s_j) = r_j$  then $z_j = f_j^{-1}(f_j(x_j) \de_{s_j} \zeta)$ has the desired properties    as $\delta_s(K)=K$.
Now  $g$ is continuous, $g(0) = 0$ and so  it suffices to show that $g(C_1 \rho_j)  \ge r_j$. 
Assume $g(C_1 \rho_j) < r_j$ and let $s  =  C_1 \rho_j$. Then there exist $\bar z_j \in  f_j^{-1}( f_j(x_j) \de_s \zeta) H$
such that $d(x_j, \bar z_j) < r_j$. By quasisymmetry  we get a contradiction:   $ d(x_j,  \bar{z}_j)<r_j=d(x_j, y_j)$, so 
 $$s=d(f_j(x_j), f_j(\bar z_j) H)  \le  d(f_j(x_j), f_j(\bar{z}_j))<C_1  d(f_j(x_j), f_j(y_j))=s.$$
   Here we are using the fact that the distortion function $\eta: [0, \infty) \to [0, \infty)$ is chosen as strictly increasing (i.e. a homeomorphism). 
\end{proof}

\bigskip
Let $X \in \fh\cap V_1$   with $\| X \| = 1$.  After passing to a subsequence and replacing $X$ by $-X$ if necessary, by 
Proposition~\ref{pr:one_sided_almost_parallel} we may assume that 
$$d(x_j \exp(tX), z_j H) \ge c d(x_j, z_j H) = c r_j   \quad \forall t \in [0, \infty)$$
where $c = c(H,G)$. 
Let $t \in [0, \frac12]$ and let $x'_j = x_j \exp(tX)$. Then $x'_j \in B(e, \frac34)$ (if 
   $j>4$  ). 
Note that since the open ball $B(x_j', c r_j)$ is disjoint from the coset $z_j H$, the image $f_j(B(x_j', c r_j))$ is disjoint from $f_j(z_jH\cap B(e,1))$, 
 which has distance at most $C_1j^{-1}r_j$ from $f_j(x'_j)$.  It follows that $\diam(f_j(B(x_j',c r_j)))\leq C_2j^{-1}r_j$ for $C_2=C_2(\eta)$.
 For large $j$, we may find a finite sequence of points $x_j=x_{j,1},\ldots,x_{j,N_j}$ of the form 
 $x_{j, i} = x_j \exp(t_i X)$ with $t_i \in [0, \frac12]$,
  where $N_j\leq c^{-1} r_j^{-1}$, $d(x_{j,k},x_{j,k+1})\leq c r_j$, and $t_{N_j} = \frac12$ so that 
  $$d(e, x_{j, N_j}) \ge  d(x_j, x_j \exp \frac12 X) - j^{-1} \ge  \frac14.$$ 
Here we have used the fact that   $d(e, \exp \frac12 X) = \frac12$ for $X \in V_1$ and $\| X\| = 1$, which follows from the definition of the Carnot distance.  
 It follows that 
\begin{align*}
d(f_j(x_j),f_j(x_{j,N_j}))&\leq \sum_k d(f_j(x_{j,k}),f_j(x_{j,k+1}))\\
&\leq c^{-1} r_j^{-1}\cdot C_2j^{-1}r_j\leq c^{-1} C_2j^{-1}\lra 0 
\end{align*}
as $j\ra \infty$.  Thus $d(e, f_j(x_{j, N_j})) \to 0$.  Since $d(e, x_{j, N_j}) \ge \frac14$ (for large $j$) we have
$f_j(B(e,1)) \subset B(e, \eta(4) d(e, f_j(x_{j,N_j}))$. This contradicts our normalization $\diam(f_j(B(e,1))=1$
     . \\

\end{proof}

\section{Structure of nonrigid Carnot groups}
\label{sec_structure_nonrigid_irreducible_first_layer}

In this section we analyze the algebraic structure of graded Lie algebras corresponding to Carnot groups which are nonrigid in the sense of Ottazzi-Warhurst.  Starting from the fact that the first layer must contain nontrivial elements with rank at most $2$ \cite{ottazzi_warhurst} (see also \cite{doubrov_radko}), we establish a trichotomy: either the first layer contains a special type of automorphism invariant subspace, or the Carnot group is $\R^n$ or a real Heisenberg group, or the Carnot  group  must have a very special structure  -- it must be a quotient of a product of real or complex Heisenberg groups by a specific type of subgroup (see Theorem~\ref{thm_irreducible_nonrigid_structure}).  This result is a variation of an unpublished classification theorem of the third author \cite{Xie_quasiconformal_on_non_rigid}.  In the concluding section we show that such product quotients admit a canonical graded product decomposition into factors whose second layers admit automorphism invariant conformal structures.

After some preliminaries, the structure theorem (Theorem \ref{thm_irreducible_nonrigid_structure})         will be given  in Subsection~\ref{subsec_irred_classification}.  We then discuss the characterization in terms of quotients in Subsection~\ref{subsec_characterization_quotients}, and in Subsection~\ref{subsec_decomposition_conformally_compact} we show that product quotients have a graded direct sum decomposition into conformal product quotients. 
\subsection{Preliminaries}

\noindent

In what follows, $\F$ will always be either $\R$ or $\C$, and if $X$ is an element of a Lie algebra $\fg$ over $\F$, then $\rank_\F X:=\rank_\F \ad_X=\dim_\F[X,\fg]$.   Given an $\F$-linear subspace $W\subset\fg$, and $r\geq 0$, we let $\rank_\F(r,W)$ and $\rank_\F(\leq r,W)$ be the collections of elements $X\in W$ with $\rank_\F X=r$, and $\rank_\F X\leq r$, respectively.  Although the field $\F$ is implicit, and therefore strictly speaking the subscript is redundant, in what follows we will sometimes have algebras over $\R$ and $\C$ in the same context, and so prefer to have the subscript to eliminate any potential ambiguity.

We denote the complexification $V\otimes\C$ of an $\R$-vector space by $V^\C$.

\begin{lemma}[Complexification of complex Lie algebras]
\label{lem_complexification_complex_lie_algebras}
Let $\fg$ be a Lie algebra over $\C$, with the complex multiplication denoted by $J$; we let $\bar\fg$ be the Lie algebra over $\C$ with the same underlying Lie algebra over $\R$, but with complex multiplication given by $-J$.  Viewing $\fg$ as a Lie algebra over $\R$,  we let $\fg^\C:=\fg\otimes_\R\C$ denote the complexification of $\fg$, considered as a vector space over $\C$ with complex multiplication denoted by $i$.  
We denote by $J^\C:=J\otimes \id_\C:\fg^\C\ra \fg^\C$  the map induced by $J$.

Then: 
\ben
\item $\fg^\C$ decomposes   as  a direct sum $\fg^\C=\fg^\C_i\oplus\fg^\C_{-i}$ of eigenspaces of $J^\C$, which are interchanged by complex conjugation.  
\item $\fg^\C_{i}$ and $\fg^\C_{-i}$ are  Lie subalgebras isomorphic over $\C$ to $\fg$ and $\bar\fg$, respectively; the isomorphisms  are  given 
   by taking twice the real part. 
\item $[\fg^\C_i,\fg^\C_{-i}]=\{0\}$.
\een
\end{lemma}
\begin{proof}
Note that 
$$
V_{\pm i}:=\{X\mp iJX\,|\, X\in \fg\}\subset \fg^\C
$$ 
is a $\C$-linear subspace contained in the $\pm i$ eigenspace of $J^\C$, and complex conjugation interchanges $V_i$ and $V_{-i}$.  Since we have $\fg^\C=V_i\oplus V_{-i}$, (1) follows.

If $Z\in \fg^\C_{\pm i}$, $Z'\in \fg^\C$, then
$$
J^\C[Z,Z']=[J^\C Z,Z']=[\pm i Z,Z']=\pm i[Z,Z']\,,
$$
so $[Z,Z']\in \fg^\C_{\pm i}$. Hence $\fg^\C_{\pm i}$ is an ideal, assertion (3) holds, and we have a direct sum decomposition $\fg^\C=\fg^\C_i\oplus\fg^\C_{-i}$.  Taking real parts gives (up to a factor) the isomorphisms $\fg^\C_{i}\simeq \fg$, $\fg^\C_{-i}\simeq \bar\fg$.
\end{proof}

\bigskip
\begin{remark}
If in Lemma~\ref{lem_complexification_complex_lie_algebras} the graded algebra $\fg$ itself happens to be the complexification of some graded algebra over $\R$ -- for instance if $\fg$ is a complex Heisenberg algebra --  then there is a $\C$-linear graded isomorphism $\fg\ra\bar\fg$ given by complex conjugation.  Therefore in this case  all four graded algebras $\fg$, $\bar\fg$, $\fg^\C_i$, $\fg^\C_{-i}$ are graded isomorphic over $\C$.
\end{remark}

\bigskip
\begin{lemma}\label{lem_rank_1_commutes_layers_geq_2}
Suppose $(\fg,\{V_j\}_{j=1}^s)$ is a graded Lie algebra over a field $\F$.  If $X\in V_1$ and $\dim_\F [X,\fg]\leq 1$, then $[X,V_j]=\{0\}$ for all $j\geq 2$.
\end{lemma}
\begin{proof}
Suppose $[X,V_1]=\{0\}$.   Since the centralizer of an element is a subalgebra and $V_1$ generates $\fg$, we get $[X,\fg]=\{0\}$ and we're done in this case.  Now suppose $[X,V_1]\neq\{0\}$. Since $\dim_\F[X, \fg] \le 1$ we deduce that
 $[X,\fg]=[X,V_1]\subset V_2$, and for all $j\geq 2$ we have $[X,V_j]\subset V_{j+1}\cap V_2=\{0\}$.
\end{proof}

\begin{lemma}\label{lem_characterization_heisenberg}
Suppose $(\fg,\{V_j\}_{j=1}^2)$ is a step $2$ graded Lie algebra over a field $\F$, such that $\dim_\F V_2= 1$ and the center of $\fg$ intersects $V_1$ trivially:  
$$
\{X\in V_1\mid [X,\fg]=\{0\}\}=\{0\}\,.
$$  
Then $\fg$ is graded isomorphic to a Heisenberg algebra over $\F$.
\end{lemma}
\begin{proof}
Identifying $V_2$ with a copy of $\F$ by an $\F$-linear isomorphism, the Lie bracket defines a skew-symmetric bilinear form $[\cdot,\cdot]:V_1\times V_1\ra \F$, which is nondegenerate.  Such forms are $\F$-linearly equivalent to a direct sum of standard $2$-dimensional symplectic forms
  by a straightforward induction argument (see for example \cite[p.349]{jacobson_basic_algebra}).
 This yields the isomorphism to the Heisenberg algebra.
\end{proof}

\bigskip
\begin{lemma}\label{lem_symplectic_transitive}
For $\F\in\{\R,\C\}$, let $\sympl(m,\F)$ be the symplectic group, i.e. the stabilizer of the standard non-degenerate skew form on $\F^{2m}$.  For  $n\geq 1$, consider an $\R$-linear subspace $K\subset (\F^{2m})^n$ which is invariant under the product action $(\sympl(m,\F))^n\acts (\F^{2m})^n$.    Then for some subset $J\subset \{1,\ldots,n\}$, the subspace $K$ is the span of the factors indexed by $J$, i.e.
$$
K=\{(x_1,\ldots,x_n)\in(\F^{2m})^n\,|\,x_k=0\,\text{ if }\,k\not\in J\}\,.
$$
\end{lemma}
\begin{proof}
Suppose for some $j\in \{1,\ldots,n\}$ we have $x_j\neq 0$ for some $x\in K$.  Choosing $T\in \sympl(m,\F)$ such that $Tx_j\neq x_j$, we get that $y\in K$, where $y_k=x_k$ if $k\neq j$, and $y_j=Tx_j$.  Then $y-x\in K$ has precisely one nonzero component, namely $Tx_j-x_j$.   Applying $\sympl(m,\F)$ and taking the linear span over $\R$, we get that the $j^{th}$ factor of $(\F^{2m})^n$ is contained in $K$.  

Applying the above to every $j\in \{1,\ldots,n\}$, the lemma follows.
\end{proof}

\bigskip
\begin{lemma}
\label{lem_complex_heisenberg_auts}
Let $\fh_m$ be the $m^{th}$ Heisenberg algebra, and $\fh_m^\C$ be the complexification, i.e. the $m^{th}$ complex Heisenberg algebra.   Let $\aut_\R(\fh_m^\C)$ and $\aut_\C(\fh_m^\C)$ denote the groups of graded $\R$-linear and $\C$-linear automorphisms, respectively.  Then $\aut_\R(\fh_m^\C)$ is generated by $\aut_\C(\fh_m^\C)$ and complex conjugation.  In particular, if $\phi:\fh_m^\C\ra\fh_m^\C$ is a  graded   $\R$-linear automorphism, then either:
\begin{itemize}
\item $\phi$ is $\C$-linear and $\det_\R\phi\restr_{V_2}>0$.
\item  $\phi$ is $\C$-antilinear and $\det_\R\phi\restr_{V_2}<0$.
\end{itemize}
\end{lemma}
\begin{proof}
The $m=1$ case appears in  \cite[Section 6]{reimann_ricci}, although their argument clearly works for general $m$.  

We give a short argument here.

Let $\fg:=\fh_m^\C$, and let $\fg=V_1\oplus V_2$ be the grading.  Viewing this as a graded algebra over $\R$, we  use Lemma~\ref{lem_complexification_complex_lie_algebras} to see that $\fg^\C=\fg^\C_i\oplus \fg^\C_{-i}$, and the first layer elements of $\C$-rank $\leq 1$ in $\fg^\C$ are precisely $V_1^\C\cap(\fg^\C_i\cup\fg^\C_{-i})$.  

If $\phi\in \aut_\R(\fg)$, then $\phi^\C:\fg^\C\ra\fg^\C$ must preserve the subset $V_1^\C\cap(\fg^\C_i\cup\fg^\C_{-i})$.  Since $\phi^\C$ is a $\C$-linear automorphism  it must preserve the collection $\{\fg^\C_{\pm i}\}$, so we have two cases: either $\phi^\C$ maps $\fg^\C_{\pm i}$ to itself, or it interchanges the two.   In the
first case $\phi^\C$ induces a $\C$-linear map from $\fg^\C_{i}$ to itself. Since the real part map $\real: \fg^\C_i \to \fg$
is $\C$-linear it follows that $\phi$ is $\C$-linear on $\fg$. In the other case $\real: \fg^\C_{-i} \to \bar\fg$ is $\C$-linear,
so $\phi:\fg\ra\bar\fg$ is $\C$-linear, which means that $\phi:\fg\ra\fg$ is $\C$-antilinear.

Finally note that an invertible $\C$-linear map on the real two-dimensional space $V_2$ has
positive determinant and an invertible $\C$-antilinear map has negative determinant.
\end{proof}

\bigskip
We now give a generalization of Lemma~\ref{lem_complex_heisenberg_auts}.

Let  $\mathfrak g=\oplus_{j=1}^s V_j$ be a Carnot  Lie algebra    over $\C$ and $I\subset \{1, \cdots, s\}$.
  Set $V_I=V_{I, \mathfrak g}:=\oplus_{j\in I}V_j$.  For $X\in \mathfrak g$,    
   let  $\text{rank}_{I,\mathfrak g}(X)=\dim_{\mathbb C}(\text{ad}_X(V_I))$.    
      Define  $r_I=r_{I, \mathfrak g}:=\min\{ \text{rank}_{I, \mathfrak g}(X)|X\in V_1\backslash \{0\}\}$  and  
$R_I=R_{I, \mathfrak g}:=\max\{ \text{rank}_{I,\mathfrak g}(X)|X\in V_1\}$.

\begin{lemma}
\label{lem_complexified_group_auts}
  Let $\mathfrak g$ be the complexification  of  a  Carnot Lie algebra.
Assume there is some $I\subset \{1, \cdots, s\}$  such that 
  $R_I< 2 r_I$. Then every graded  $\R$-linear  automorphism  of  
  $\mathfrak g$ is either $\mathbb C$-linear or $\mathbb C$-antilinear.  
\end{lemma}

\begin{proof}
By Lemma \ref{lem_complexification_complex_lie_algebras},  the complexification  $\mathfrak g^{\mathbb C} $  of $\mathfrak g$
 admits a  decomposition $\mathfrak g^{\mathbb C} =\mathfrak g_i^{\mathbb C}\oplus \mathfrak g_{-i}^{\mathbb C}$ into a  direct sum of two graded 
  subalgebras   over $\mathbb C$ and 
 $[\mathfrak g_i^{\mathbb C},   \mathfrak g_{-i}^{\mathbb C}]=0$. 
  Let $\phi\in \text{Aut}_{\mathbb R}(\mathfrak g)$.
     Then  $\phi^{\mathbb C}: \mathfrak g^{\mathbb C} \ra \mathfrak g^{\mathbb C} $  is a graded  $\mathbb C$-linear automorphism.   We shall show that 
   $\phi^{\mathbb C}$ either 
    maps $\mathfrak g_{\pm i}^{\mathbb C}$ to itself or switches the two.   The lemma then follows by the argument in the proof of Lemma~\ref{lem_complex_heisenberg_auts}. 

Let  $0\not=X\in V_{1,\mathfrak g_i^{\mathbb C}}\cup V_{1,\mathfrak g_{-i}^{\mathbb C}}$.
  Since  both  $ \mathfrak g_i^{\mathbb C}$ and $ \mathfrak g_{-i}^{\mathbb C}$  are isomorphic to 
  $\mathfrak g$ and $[\mathfrak g_i^{\mathbb C},   \mathfrak g_{-i}^{\mathbb C}]=0$, we have
     $\text{rank}_{I, \mathfrak g^{\mathbb C}}(X)\le R_I<2 r_I$.  On the other hand,   we have   $\text{rank}_{I,  \mathfrak g^{\mathbb C}}(X)\ge 2 r_I$
       for any
      $X\in V_{1, \mathfrak g^{\mathbb C}}\backslash (V_{1, \mathfrak g_i^{\mathbb C}}\cup V_{1,\mathfrak g_i^{\mathbb C}})$. To see this,   we write $X=X_++X_-$ with $0\not=X_+\in V_{1, \mathfrak g_i^{\mathbb C}}$, $0\not=X_-\in V_{1, \mathfrak g_{-i}^{\mathbb C}}$.  Then 
       $$\text{ad}_X(V_{I, \mathfrak g^{\mathbb C}})
       =\text{ad}_{X_+}(V_{I, \mathfrak g_i^{\mathbb C}})\oplus 
       \text{ad}_{X_-}(V_{I, \mathfrak g_{-i}^{\mathbb C}})$$  and so 
       $$\text{rank}_{I, \mathfrak g^{\mathbb C}}(X)=\text{rank}_{I,\mathfrak g_i^{\mathbb C}} (X_+)
       +\text{rank}_{I,\mathfrak g_{-i}^{\mathbb C}} (X_-)\ge 2 r_I.$$
        Here we used the facts that  
 $ \mathfrak g_i^{\mathbb C}$ and $ \mathfrak g_{-i}^{\mathbb C}$  are isomorphic to 
  $\mathfrak g$ and $[\mathfrak g_i^{\mathbb C},   \mathfrak g_{-i}^{\mathbb C}]=0$.
   Since  for any graded    $\mathbb C$-linear   isomorphism  
     $f: \mathfrak g_1 \ra \mathfrak g_2$   of   complex Carnot algebras  
         we have 
   $\text{rank}_{I, \mathfrak g_1}(X)=\text{rank}_{I, \mathfrak g_2}(f(X))$, we see that 
   $\phi^{\mathbb C}$  maps  $V_{1, \mathfrak g_i^{\mathbb C}}\cup V_{1, \mathfrak g_{-i}^{\mathbb C}}$ to itself.  Since  $\phi^{\mathbb C}$ is $\mathbb C$-linear,   $\phi^{\mathbb C}$   either 
    maps $\mathfrak g_{\pm i}^{\mathbb C}$ to itself or switches the two. 
\end{proof}

\bigskip
 The assumption in Lemma \ref{lem_complexified_group_auts}  is satisfied     for the
complexification   $\mathfrak g$   of  the   following classes of Carnot algebras:\newline
(1)  model filiform algebras; in this case,  $R_{\{1\}}=r_{\{1\}}=1$.   \newline
(2) free nilpotent Lie   algebras;  more generally, quotients of free  nilpotent Lie  algebras by graded ideals contained in the direct sum of  higher layers  $V_j$, $j\ge 3$;   in this case, $R_{\{1\}}=r_{\{1\}}>0$.\newline
(3)  Carnot algebras $\mathfrak h=\oplus_j V_j$  satisfying   $[X, V_i]=V_{i+1}$  for  all $0\not=X\in V_1$     and   some  fixed $i\ge 1$   with $\text{dim}(V_{i+1})$ odd.   These include 
nilpotent Lie algebras satisfying M\'etivier's hypothesis (H) (in particular,    $H$ type algebras)  whose centers have odd dimension.   To see that $\mathfrak g:=\mathfrak h^{\mathbb C}=\oplus_j V_j^{\mathbb C}$ satisfies the assumption of   Lemma \ref{lem_complexified_group_auts},  we denote 
 $\text{dim}_{\mathbb R}(V_{i+1})=2k+1$. Then 
  $\text{dim}_{\mathbb C}(V^{\mathbb C}_{i+1})=2k+1$  and so $R_{\{i\}}\le  2k+1$.   
   We show that  $\text{rank}_{\{i\},\mathfrak g}(X)\ge k+1$ for any $0\not=X\in V^{\mathbb C}_1$, which implies  $r_{\{i\}}\ge  k+1$  and so $R_{\{i\}}<2r_{\{i\}}$.     Write $X=X_1+iX_2$ with $X_1, X_2\in V_1$. We may assume $X_1\not=0$. 
We have $\real[X, V^{\mathbb C}_i]\supset \real[X, V_i]=V_{i+1}.$  Hence 
$\text{dim}_{\mathbb R}([X, V^{\mathbb C}_i])\ge 
\text{dim}_{\mathbb R}(V_{i+1})=2k+1$, which implies  $\text{rank}_{\{i\},\mathfrak g}(X)=\text{dim}_{\mathbb C}([X, V^{\mathbb C}_i])\ge   k+1$.

\bigskip

\bigskip
\bigskip
\begin{lemma}\label{lem_rank_1_span_step_2}
Let $(\fg,\{V_j\}_{j=1}^2)$ be a step $2$ graded Lie algebra over $\F$.  Suppose $\rank_\F(0,V_1)=\{0\}$, and $\Span_\F(\rank_\F(1,V_1))=V_1$.  Then there is a collection $\fg_1,\ldots,\fg_n$ of graded subalgebras (over $\F$) of $\fg$ such that:
\ben
\item Each $\fg_j$ is graded isomorphic over $\F$ to some Heisenberg algebra over $\F$.  
\item The first layers of the $\fg_j$s define a direct sum decomposition of $V_1$:
$$
V_1=\oplus_j(V_1\cap \fg_j)\,.
$$
\item The $\fg_j$s commute with one another:  $[\fg_j,\fg_k]=0$ for $1\leq j\neq k\leq n$.
\item  $\fg_j\cap \fg_k=\{0\}$ for $1\leq j\neq k\leq n$.
\item The collection is permuted by the graded automorphism group $\aut(\fg)$.
\een
Moreover conditions (1)-(4) determine $n$ and the collection $\fg_1,\ldots,\fg_n$ uniquely. 
\end{lemma}

\begin{proof}  
Let $\{L_j\}_{j\in J}$ be the collection of $1$-dimensional subspaces of $V_2$ of the form $[X,\fg]$, where 
$X \in \rank_\F(1, V_1)$; here the index set $J$ might be infinite a priori.   For every $j\in J$, let $K_j:=\{X\in V_1\mid [X,\fg]\subset L_j\}$.  Then $K_j$ is a subspace of $V_1$, and by our assumption that $\Span_\F\rank_\F(1,V_1)=V_1$, the $K_j$s span $V_1$.  

Note that if $j\neq j'$ then $[K_j,K_{j'}]\subset L_j\cap L_{j'}=\{0\}$. Therefore for every $j_0\in J$ we have  $[K_{j_0},\sum_{j\neq j_0}K_j]=\{0\}$.   It follows that if  $X\in K_{j_0}\cap \sum_{j\neq j_0}K_j$ then $[X,K_j]=\{0\}$ for all $j$, and so $[X,\fg]=\{0\}$, forcing $X=0$ by assumption.  Hence we have a direct sum decomposition $V_1=\oplus_j K_j$, and in particular $J$ is finite.

For every $j\in J$, let $\fg_j:=K_j\oplus L_j$.  Since $\fg$ has step $2$, $\fg_j$ is a graded subalgebra of $\fg$.  Since $V_1\setminus\{0\}$ has no rank zero elements, and $[\fg_j,\fg_k]=\{0\}$ for $k\neq j$, it follows that $\fg_j$ has no rank zero first layer elements.  By Lemma~\ref{lem_characterization_heisenberg},  $\fg_j$ is isomorphic to a Heisenberg  algebra over $\F$, for every $j$.  

If $j\neq k$, we get $\fg_j\cap \fg_k=\{0\}$ from the fact that its projections to both layers are  $\{0\}$.  

To prove uniqueness, we observe that if $X\in V_1\setminus\cup_j\fg_j$, then $X$ has rank at least $2$.  Thus if $\fg_1',\ldots,\fg_{k'}'$ is another collection of subalgebras satisfying (1)-(4), then each $\fg_{j'}'\cap V_1$ must be contained in $\fg_j\cap V_1$ for some $1\leq j\leq n$, and vice-versa.  Since the $\fg_j$s are determined by their first layers, this gives uniqueness and consequently assertion (5) as well.
\end{proof}

\subsection{The classification}\label{subsec_irred_classification}

 We recall that from \cite{ottazzi_warhurst,doubrov_radko}, a Carnot group $G$ is nonrigid if and only if  the first layer of its complexification contains an element $X\neq 0$ with $\rank_\C X\leq 1$.  Hence the following theorem yields a dichotomy for nonrigid graded Lie algebras: either the first layer contains a special type of automorphism invariant subspace, or the graded algebra has a very special structure. 

\begin{theorem}\label{thm_irreducible_nonrigid_structure}
Let $(\fg,\{V_j\}_{j=1}^s)$ be a (real) graded Lie algebra, and $(\fg^\C,\{V_j^\C\}_{j=1}^s)$ be the complexification with its induced grading.
Suppose:
\begin{enumerate}[label=(\it{\alph*})]
\item There is no $\aut(\fg)$-invariant subspace $\{0\}\neq W\subsetneq V_1$ such that $[W,V_i]=\{0\}$ for all $i\geq 2$. 
\item There is a nonzero element $Z\in (V_1^\C)\setminus\{0\}$ such that $\rank_\C Z\leq 1$ (or equivalently, $(\fg,\{V_j\}_{j=1}^s)$ is the graded Lie algebra of a nonrigid Carnot group).
\end{enumerate}
Then either $\fg$ is abelian, or it has step $2$, and for some  $\F\in\{\R,\C\}$, $n\geq 1$, there is a collection $\fg_1,\ldots,\fg_n$ of graded subalgebras of $\fg$ with the following properties:
\ben
\item  \label{it:isomorphic_mth_heisenberg} For some $m$, each $\fg_j$ is graded isomorphic over $\R$ to the $m$-th Heisenberg algebra   over $\F$ (viewed as a graded Lie algebra over $\R$).  
\item The first layers of the $\fg_j$s define a direct sum decomposition of $V_1$:
$$
V_1=\oplus_j(V_1\cap \fg_j)\,.
$$
\item The $\fg_j$s commute with one another:  $[\fg_j,\fg_k]=0$ for $1\leq j\neq k\leq n$.
\item  \label{it:complex_not_distinct} If $\F=\R$, then the second layers are distinct:  $\fg_j\cap V_2\neq\fg_k\cap V_2$ for $1\leq j\neq k\leq n$.  If $\F=\C$, then the second layers need not be distinct.  However, for each $j$ we have a (graded) decomposition of the complexification $\fg_j^\C=(\fg_j^\C)_i\oplus(\fg_j^\C)_{-i}$  from Lemma~\ref{lem_complexification_complex_lie_algebras};
the second layers of $(\fg_j^\C)_{\pm i}$ are distinct and interchanged by 
complex conjugation, and we obtain distinct pairs of second layers as $j$ varies. 
\item $\aut(\fg)$ preserves the collection $\{\fg_j\}_{j=1}^n$, and permutes the $\fg_j$s transitively.
\een
Moreover the field $\F$, and the collection $\fg_1,\ldots,\fg_n$ are uniquely determined by $\fg$.
\end{theorem} 

\bigskip
\begin{remark}
Note that in \eqref{it:complex_not_distinct} for $\F = \C$ the definition of $(\fg_j^\C)_{\pm i}$ requires a complex multiplication 
on $\fg_j$. Such a multiplication is induced by a choice of an $\R$-linear graded isomorphism from the $m$-th complex Heisenberg group to $\fg_j$ which exist by \eqref{it:isomorphic_mth_heisenberg}. In view of 
Lemma~\ref{lem_complex_heisenberg_auts}, different choices of an $\R$-linear graded isomorphism yield the same complex
structure on $\fg_j$, up to a possible change of sign. Changing the sign of the complex structure just exchanges
$(\fg_j^\C)_{i}$ and $(\fg_j^\C)_{-i}$. Thus the collection
 $\{ (\fg_j^\C)_{i}, (\fg_j^\C)_{-i} \}$ is uniquely determined by $\fg_j$.
\end{remark}

\bigskip
\begin{remark}
The reader may wonder why Hypothesis (a) in Theorem~\ref{thm_irreducible_nonrigid_structure} includes the restriction that $[W,V_i]=\{0\}$ for all $i\geq 2$.  The theorem would remain true if this restriction is dropped, since, due to the negation, it would make this a logically stronger hypothesis; however, for the results in Section~\ref{sec_qs_rigidity_carnot_groups} we need the stronger result of Theorem~\ref{thm_irreducible_nonrigid_structure}.
\end{remark}

\bigskip

\begin{proof}

We first sketch the overall logic before beginning the formal proof. 

By our assumptions, it follows readily that $V_1\setminus\{0\}$ contains an element $X$ with $\rank_\R(X)\in \{0,1,2\}$.  We then deal with the three possibilities by elimination.  If $\rank_\R(0,V_1)\neq\{0\}$, then one concludes that $\fg$ is abelian.  Assuming $\rank_\R(0,V_1)=\{0\}$ and $\rank_\R(1,V_1)\neq\emptyset$, we reduce to Lemma~\ref{lem_rank_1_span_step_2}.  Finally, assuming $\rank_\R(\leq 1,V_1)=\{0\}$, we analyze the situation by complexifying, applying Lemma~\ref{lem_rank_1_span_step_2} to the complexification, and interpreting the results back in the original graded algebra $\fg$.

We now return to the proof.

Suppose $\rank_\R(0,V_1)\neq\{0\}$.  Then $\rank_\R(0,V_1)$
is a nontrivial, $\aut(\fg)$-invariant subspace of $V_1$ which commutes with $\fg$.
Thus by hypothesis (a) we have $\rank_\R(0, V_1) = V_1$.
   
 Since $V_1$ generates $\fg$, it follows that $\fg$ is abelian, and we are done.  Therefore we may assume that 
\begin{equation}\label{eqn_rank_geq_1}
\rank_\R (0,V_1)=\{0\}\,.
\end{equation}

Now suppose $\rank_\R(1,V_1)\neq\emptyset$.  Then $W_1:=\Span_\R \rank_\R(1,V_1)$ is a nontrivial, $\aut(\fg)$-invariant subspace of $V_1$. 
By Lemma~\ref{lem_rank_1_commutes_layers_geq_2} we have $[W_1,V_j]=\{0\}$ for all $j\geq 2$. Thus hypothesis (a) implies that $W_1 = V_1$.
Using Lemma~\ref{lem_rank_1_commutes_layers_geq_2} again we see that $[V_1, \oplus_{j\ge 2} V_j] = \{0\}$ and we conclude that $\fg$ has step 
$2$.

By Lemma~\ref{lem_rank_1_span_step_2} we get graded subalgebras $\fg_1,\ldots,\fg_n$ satisfying (1)-(4) of  Lemma~\ref{lem_rank_1_span_step_2}   in this case, and by (5) of Lemma~\ref{lem_rank_1_span_step_2} they are permuted by $\aut(\fg)$.   
 By hypothesis (a) they must be permuted transitively, since otherwise an orbit would give rise to a nontrivial $\aut(\fg)$-invariant subspace of $V_1$ which commutes with $\oplus_{j \ge 2} V_j$.   Thus (5) holds.  In particular, all the $\fg_j$ are isomorphic to the same Heisenberg algebra and we are done in this case.

We now assume in addition that $\rank_\R(1,V_1)=\emptyset$, i.e.
\begin{equation}\label{eqn_rank_geq_2}
\dim_\R[X,\fg]\geq 2\quad\text{for every}\quad X\in V_1\setminus\{0\}\,.
\end{equation}

Recall that $\rank_\C(1,V_1^\C)$ is the collection of $Z\in V_1^\C$ such that $\rank_\C Z=1$.

\begin{claim}
Suppose $0\neq Z\in \rank_\C(\leq 1,V_1^\C)\subset \fg^\C$, and let
\begin{align*}
P_Z^\R&:=\Span_\R(\real Z,\imag Z)\\
P_Z^\C&:=\Span_\C(Z,\bar Z)=\C Z+\C\bar Z\,.
\end{align*}
  Then
\ben
\item $\dim_\C[Z,\fg^\C]=1$.
\item  $P_Z^\R\otimes\C=P_Z^\C$.
\item  $\dim_\R P_Z^\R=\dim_\C P_Z^\C=2$.
\item $[\bar Z,\fg^\C]\neq [Z,\fg^\C]$.
\item $\dim_\R[P_Z^\R,\fg]
=\dim_\C[P_Z^\C,\fg^\C]=2$.
\een
\end{claim}
\begin{proof}[Proof of claim]
We cannot have $\dim_\C[Z,\fg^\C]=0$, since then we would have $\real Z,\imag Z\in V_1$ and $[\real Z,\fg]= [\imag Z,\fg]=0$, contradicting (\ref{eqn_rank_geq_1}).  Hence $\dim_\C[Z,\fg^\C]=1$ and (1) holds.

(2) is immediate from $\real Z=\frac12 (Z+\bar Z)$, $\imag Z=\frac{1}{2i}(Z-\bar Z)$.

Note that $P_Z^\C=P_Z^\R\otimes \C$ implies $\dim_\R P_Z^\R=\dim_\C P_Z^\C$.  We cannot have $\dim_\C P_Z^\C=1$, since then $P_Z^\C$ would contain a nonzero real element with complex rank $1$, and hence real rank $1$, contradicting (\ref{eqn_rank_geq_2}).  Thus $\dim_\C P_Z^\C=2$, and (3) holds.

We now consider the action of complex conjugation $Z \mapsto \bar Z$.  Suppose $[\bar Z,\fg^\C]=[Z,\fg^\C]$.  Then $P_Z^\C$ is a $\C$-linear subspace of $V_1^\C$ that is complex conjugation invariant, and $[P_Z^\C,\fg^\C]=[Z,\fg^\C]=[\bar Z, \fg^\C]$ is a $\C$-linear, complex conjugation invariant 
subspace with $\dim_\C[P_Z^\C,\fg^\C]=1$.  Therefore the real points $P_Z^\R=\real P_Z^\C$ have the property that $[P_Z^\R,\fg]$ lies in the real points of $[P_Z^\C,\fg^\C]$, which form an $\R$-linear subspace of dimension $1$ over $\R$.  This contradicts (\ref{eqn_rank_geq_2}).  It follows that $[\bar Z,\fg^\C]\neq[Z,\fg^\C]$, so (4) holds.

 Finally, (4) implies that $\dim_\C[P_Z^\C,\fg^\C]=2$. Since $[P_Z^\R,\fg]\otimes\C\subset [P_Z^\C,\fg^\C]$, we get $\dim_\R[P_Z^\R,\fg]\leq\dim_\C[P_Z^\C,\fg^\C]=2$.  By (\ref{eqn_rank_geq_2}) we cannot have $\dim_\R[P_Z^\R,\fg]\leq 1$, so (5) follows.
 \end{proof}

\medskip
Let $\hat W_1:=\Span_\R\{P_Z^\R\,|\, Z\in \rank_\C(1,V_1^\C)\}$.  This is a nontrivial $\aut(\fg)$ invariant  subspace of $V_1$.
If  $Z\in \rank_\C(1,V_1^\C)$, then for $j\geq 2$ we have
$$
[P_Z^\R,V_j]\subset [P_Z^\C,V_j^\C]=[Z,V_j^\C]+[\bar Z,V_j^\C]=\{0\}
$$
by Lemma~\ref{lem_rank_1_commutes_layers_geq_2} and the identity $\dim_\C[\bar Z, \fg^\C] = \dim_\C[Z, \fg^\C]$.
Thus by hypothesis (a) we get $\hat W_1 = V_1$. 
This implies that $\fg$ is a step $2$ graded algebra, and hence $\fg^\C$ also has step $2$.  Furthermore, 
\begin{align*}
V_1^\C&=V_1\otimes\C=\hat W_1\otimes\C=\Span_\C\{P_Z^\R\otimes\C\mid Z\in \rank_\C(1,V_1^\C)\}\\
&=\Span_\C\{P_Z^\C\mid Z\in \rank_\C(1,V_1^\C)\}\,.
\end{align*}
Thus
  $\rank_\C(1,V_1^\C)$ spans $V_1^\C$, and we can apply Lemma~\ref{lem_rank_1_span_step_2} to obtain a collection $\{\hat\fg_j\}_{j\in J_0}$ of graded subalgebras (over $\C$) of $\fg^\C$ as in that lemma. 
   By the uniqueness assertion in  Lemma~\ref{lem_rank_1_span_step_2}, the collection $\{\hat\fg_j\}_{j\in J_0}$ is preserved by both complex conjugation and the  action  $\aut(\fg)\acts \fg^\C$.
  Hence by (4) of the claim, complex conjugation will act freely on the collection $\{V_2^\C\cap \hat\fg_j\}$ of second layers (here we are using the fact that $[\hat\fg_j,\fg^\C]=[\hat\fg_j,V_1^\C]=[\hat\fg_j,\hat\fg_j]=V_2^\C\cap\hat\fg_j$).  Picking one representative from each orbit under complex conjugation, we get a subcollection $\{\hat\fg_j\}_{j\in J}$ such that we get a direct sum decomposition
$$
V_1^\C=\oplus_{j\in J}[(\hat\fg_j\oplus \ol{\hat\fg_j})\cap V_1^\C]
$$
and also for every $j\in J$ we have 
\begin{equation}  \label{fg_bar_fg_disjoint}
\hat\fg_j\cap\ol{\hat\fg_j}=\{0\}
\end{equation} 
by (4) of the claim.   For $j\in J$, define 
\begin{equation}
\label{eqn_fg_j}
\fg_j:=\real(\hat\fg_j\oplus\ol{\hat\fg_j})=\{X+\bar X\mid X\in \hat\fg_j\}\,;
\end{equation}
here we have used the fact that the real subspace is the $\R$-linear subspace of vectors which are fixed under complex conjugation. 
In view of  \eqref{fg_bar_fg_disjoint} the map $X \to X + \bar X$ is an $\R$-linear graded automorphism from $\hat \fg_j $ to $\fg_j$. 
 Using this graded isomorphism, we may pass the complex multiplication on $\hat\fg_j$ to $\fg_j$; then by applying Lemma~\ref{lem_complexification_complex_lie_algebras}, 
 we have identifications $\hat\fg_j=(\fg_j^\C)_i$, $\ol{\hat\fg}_j=(\fg_j^\C)_{-i}$.  
 Note that we have direct sum decomposition $V_1=\oplus_{j\in J}(\fg_j\cap V_1)$.  
 Since the action $\aut(\fg)\acts \fg^\C$ preserves
  the collection $\{ \hat\fg_j   \}_{j\in  J_0}$  
 and commutes with complex conjugation,
  it follows from (\ref{eqn_fg_j}) that $\aut(\fg)\acts\fg$ preserves $\{\fg_j\}_{j\in J}$.
 By hypothesis (a),  $\aut(\fg)$ permutes the $\fg_j$ transitively. Indeed a non-trivial orbit would generate a non-trivial $\aut(\fg)$ invariant
 subspace of $V_1$. This subspace trivially commutes with $V_2$ since we have already shown that $\fg$ has step $2$. Transitivity of the action implies that 
the $\fg_j$s are isomorphic to the $m$-th complex Heisenberg algebra for a fixed $m$ independent of $j$.   Assertions (1)-(4) now follow from the above discussion.

Thus we have established the existence of $\F$, $n$, and the collection $\{\fg_j\}_{j=1}^n$ satisfying (1)-(5).  We now prove uniqueness.  Suppose $\F'$, $n'$, and $\{\fg_j'\}_{j=1}^{n'}$  also satisfy (1)-(5).   Note that if either $\F=\R$ or $\F'=\R$, then $\rank_\R(1,V_1)\neq\emptyset$, while if $\F=\C$ or $\F'=\C$ then $\rank_\R(1,V_1)=\emptyset$ because $\dim_\R[X,\fg] \geq 2$ for every $X\in V_1\setminus\{0\}$.  Hence we must have $\F'=\F$.  
If $\F'=\F=\R$, then Lemma~\ref{lem_rank_1_span_step_2} implies that the collections of subalgebras $\{\fg_j\}$, $\{\fg_j'\}$ coincide.  If $\F'=\F=\C$, then by looking at the complexification $\fg^\C$, and using condition (4), we may argue as in the proof of Lemma~\ref{lem_rank_1_span_step_2} to see that the collections of pairs $\{(\fg^\C_j)_{\pm i}\}$, $\{(\fg^{'\C}_j)_{\pm i}\}$ coincide.  Taking real parts we see that the collections $\{\fg_j\}$, $\{\fg_j'\}$ are the same. 
\end{proof}

\bigskip
\subsection{A characterization using quotients}\label{subsec_characterization_quotients}
\mbox{}We now show that there is an alternative description of the graded algebras 
which satisfy hypotheses of Theorem~\ref{thm_irreducible_nonrigid_structure}  as a certain type of quotient of a product.

Let $(\fg,\{V_j\}_{j=1}^2)$ be a nonabelian graded Lie algebra as in Theorem~\ref{thm_irreducible_nonrigid_structure}, and let $\F$ and $\fg_1,\ldots,\fg_n$ be the uniquely determined graded subalgebras as supplied by that theorem.  

Let $(\tilde\fg,\{\tilde V_j\}_{j=1}^2):=(\oplus_{k=1}^n\fg_k,\{\oplus_{k=1}^nV_{j,k}\}_{j=1}^2)$, where $\fg_k=V_{1,k}\oplus V_{2,k}$ is the layer decomposition of $\fg_k$; here $\oplus_{k=1}^n \fg_k$ denotes the abstract direct sum (the subalgebras are typically not  independent in $\fg$).  Then we get an epimorphism of graded Lie algebras $\pi:\tilde\fg\ra \fg$ by sending $\fg_k$ to $\fg$ by the inclusion.  Since the $\fg_k$s are unique, they are permuted by $\aut(\fg)$; hence we get an induced action $\aut(\fg)\acts \tilde\fg$, and $\pi$ is $\aut(\fg)$-equivariant.

Let $K:=\ker(\pi)$. Since $\pi$ is $\aut(\fg)$-equivariant, it follows that $K$ is $\aut(\fg)$-invariant.  To summarize: 

\begin{lemma}
\label{lem_autg_lifts_to_auttildeg}
We have a canonical embedding $\aut(\fg)\hookrightarrow\aut(\tilde\fg)$, and the image is precisely $\stab(K,\aut(\tilde\fg))$, the stabilizer of $K$ in $\aut(\tilde\fg)$.
\end{lemma}
Henceforth we will sometimes identify $\aut(\fg)$ with $\stab(K,\aut(\tilde\fg))$.

The first layers $\fg_j\cap V_1$ yield a direct sum decomposition of $V_1$, and therefore $K\subset \tilde V_2$.   Taken together, the conditions (2)-(5)  from Theorem~\ref{thm_irreducible_nonrigid_structure} imply that the following properites hold for the projection $\pi:\tilde\fg\ra\fg=\tilde\fg/K$:
\ben
\item The restriction of $\pi$ to $\fg_k$ is injective.
\item  If $\F=\R$, then for $j\neq k$, the second layers $V_{2,j}$, $V_{2,k}$ project under $\pi$ to distinct subspaces of $V_2\simeq \tilde V_2/K$.   
If $\F=\C$, then for each $j$ we have a (graded) decomposition of the complexification $$\fg_j^\C=(\fg_j^\C)_i\oplus(\fg_j^\C)_{-i}=(V^\C_{1,j})_i\oplus (V^\C_{2,j})_i\oplus(V^\C_{1,j})_{-i}\oplus(V^\C_{2,j})_{-i}$$  given by Lemma~\ref{lem_complexification_complex_lie_algebras}; the projections  $\pi((V^\C_{2,j})_i)$, $\pi((V^\C_{2,j})_{-i})$ are distinct subspaces which are interchanged by complex conjugation, and yield distinct pairs as $j$ varies. 
\item The action $\aut(\fg)\acts \tilde\fg$ permutes the summands $\{\fg_k\}$ transitively.
\een

The converse holds:
\begin{lemma}
\label{lem_converse_product_quotient}
Suppose $\F\in\{\R,\C\}$, $m\geq 1$, and we are given the graded direct sum $\tilde \fg=\oplus_{k=1}^n\tilde\fg_k$, where the $\tilde\fg_k$s are graded isomorphic over $\R$ to the $m$-th Heisenberg algebra over $\F$, and an $\R$-linear subspace $K\subset \tilde V_2$, such that the following hold:

\begin{enumerate}[start= 4, label=(\arabic*)]
\item \label{eqn_k_cap_fg_k} $K\cap \tilde\fg_k=\{0\}$ for all $1\leq k\leq n$.
\item \label{distinct_2nd_layers} If $\F=\R$ then we have $K+(\tilde\fg_j\cap \tilde V_2)\neq K+(\tilde\fg_k\cap \tilde V_2)$ for $j\neq k$. If $\F=\C$, then for each $j$ we have a (graded) decomposition of the complexification $\fg_j^\C=(\fg_j^\C)_i\oplus(\fg_j^\C)_{-i}=(V^\C_{1,j})_i\oplus (V^\C_{2,j})_i\oplus(V^\C_{1,j})_{-i}\oplus(V^\C_{2,j})_{-i}$ given by Lemma~\ref{lem_complexification_complex_lie_algebras};  the projections  $\pi((V^\C_{2,j})_i)$, $\pi((V^\C_{2,j})_{-i})$ are distinct subspaces  which are interchanged by complex conjugation, and yield  distinct pairs as $j$ varies.
\item \label{eqn_stabilizer_acts_transitively} The stabilizer of $K$ in $\aut(\tilde\fg)$  permutes the factors $\{\tilde\fg_k\}$ transitively.
\end{enumerate}

Then the quotient $\fg:=\tilde \fg/K$ with the grading $\{\tilde V_1,\tilde V_2/K\}$ induced from $\tilde\fg$ satisfies the hypotheses of Theorem~\ref{thm_irreducible_nonrigid_structure}.  
\end{lemma}
\begin{proof}
Let $\pi:\tilde\fg\ra \fg$ be the quotient map.

Suppose $X\in V_1\setminus\{0\}$.  Pick $\tilde X\in \tilde V_1$ with $\pi(\tilde X)=X$, and let $\tilde X_1+\ldots+\tilde X_n$ be the decomposition induced by $\tilde \fg=\oplus_k\tilde\fg_k$.  Then $\tilde X_k\neq 0$ for some $k$, and there is a $\tilde Y\in \tilde\fg_k$ such that $[\tilde X,\tilde Y]\in V_{2,k}\setminus\{0\}$.  Now $[X,\pi(\tilde Y)]=\pi([\tilde X,\tilde Y])=\pi([X_k,\tilde Y])\neq 0$ by \ref{eqn_k_cap_fg_k}.  Therefore $V_1$ contains no nonzero elements of rank zero.

If $\F=\R$, then $\rank_\R(1,\tilde V_1)$ spans $\tilde V_1$, so $\rank_\C(1,V_1^\C)$ spans $V_1^\C$.  If $\F=\C$, then by Lemma~\ref{lem_complexification_complex_lie_algebras} we have the decomposition
$$
\fg^\C=\tilde \fg^\C/K^\C=(\oplus_{k=1}^n\tilde\fg_k^\C)/K^\C
=(\oplus_{k=1}^n[(\tilde\fg_k^\C)_i\oplus(\tilde\fg_k^\C)_{-i}])/K^\C\,,
$$
so $\rank_\C(1,V_1^\C)$ also spans $V_1^\C$ in this case.

To see that $\aut(\fg)$ acts irreducibly on the first layer suppose that $W \subset V_1$ is a nontrivial subspace invariant under $\aut(\fg)$ and consider first the case $\F = \R$. 
Then in particular $W$ is invariant under the action of the subgroup of $\aut(\tilde\fg)$ consisting of     elements that act trivially on $V_2$. So we may apply Lemma~\ref{lem_symplectic_transitive} to conclude that $W=\oplus_{j\in J}V_{1,j}$ for some nonempty subset $J\subset \{1,\ldots,n\}$.  
Now \ref{eqn_stabilizer_acts_transitively} implies that $J=\{1,\ldots,n\}$. 
Now consider the case $\F = \C$. Then the map $\Phi$ which acts as complex multiplication   by $i$    on the first layer and 
multiplication by $-1$ on the second layer belong to $\aut(\tilde \fg)$. Since it preserves  $K$ it also defines
an element of $\aut(\fg)$. Thus $W$ is a $\C$-linear subspace of $V_1$. Now we can conclude as before by
considering the subgroup of $\aut(\tilde\fg)$ consisting of $\C$-linear elements that act trivially on $V_2$.
Thus $(\fg,\{V_j\}_{j=1}^2)$ is a nonrigid graded algebra such that $\aut(\fg)$ acts  irreducibly on the first layer; in particular hypothesis (a) of Theorem~\ref{thm_irreducible_nonrigid_structure} holds.   This concludes the proof. 
\end{proof}

\bigskip
\begin{definition}
\label{def_product_quotient}
A {\em product quotient} is a Carnot group $G$ of the form $\tilde G/\exp(K)$, where the graded Lie algebra $\tilde\fg$ of $\tilde G$ and $K\subset \tilde\fg$ satisfy the assumptions of Lemma~\ref{lem_converse_product_quotient}, and in the case $\F=\R$ we have $n\geq 2$.  We will  use the term product quotient to refer to both Carnot groups and their graded Lie algebras.
\end{definition}

\begin{remark}
\label{rem_product_quotient_same_as_classification}
By virtue of Theorem~\ref{thm_irreducible_nonrigid_structure} and Lemma~\ref{lem_converse_product_quotient}, product quotients are precisely the Carnot groups whose associated graded algebra satisfies the hypotheses of Theorem~\ref{thm_irreducible_nonrigid_structure}, with the exception of  abelian groups and real Heisenberg groups.  We have chosen to exclude the latter two cases from the definition because our main objective is to study rigidity phenomena which fail to hold in those cases.  We emphasize that a complex Heisenberg group is a product quotient, while a real Heisenberg group is not.
\end{remark}

\bigskip\bigskip 
We close this subsection by discussing a few examples of product quotients. 

The simplest examples of product quotients are simply products, i.e. they are of the form $\fg=\tilde\fg=\oplus_{j=1}^n\tilde\fg_j$, $K=\{0\}$, where each summand $\tilde\fg_j$ is a copy of a fixed Heisenberg algebra over $\F\in\{\R,\C\}$, and $n\geq 2$ if $\F=\R$.

\begin{example}
\label{ex_diagonal_product_quotient} (Diagonal product quotient) Let $\tilde\fg=\oplus_{j=1}^n\tilde\fg_j$, where $n\geq 3$, $\tilde\fg_j$ is a copy of the first real Heisenberg group, with basis $\tilde X_{3j-2}, \tilde X_{3j-1}, \tilde X_{3j}$, $[\tilde X_{3j-2},\tilde X_{3j-1}]=-\tilde X_{3j}$.  Let $K\subset \tilde V_2$ be the diagonal subspace $K=\Span(\tilde X_3+\ldots+\tilde X_{3n})$.  Then the permutation action $S_n\acts\tilde\fg$ leaves $K$ invariant, and hence  Lemma~\ref{lem_converse_product_quotient}\ref{eqn_stabilizer_acts_transitively}  holds.  We may obtain similar examples by using the $m$-th real or complex Heisenberg group, instead of the first real Heisenberg group.
\end{example}

\bigskip
We emphasize that it is necessary to have $n\geq 3$ in Example~\ref{ex_diagonal_product_quotient} 
for the conditions in Lemma~\ref{lem_converse_product_quotient} to be satisfied, because otherwise Lemma~\ref{lem_converse_product_quotient}\ref{distinct_2nd_layers} would be violated.  We also point out that if $n=2$ and $K\subset \tilde V_2$ is the diagonal subspace, then:
\bit
\item If the $\tilde\fg_j$s are copies of the first real Heisenberg group, then $K$ does not satisfy the assumptions of Lemma~\ref{lem_converse_product_quotient},  the quotient of $\tilde\fg/K$ is the second real Heisenberg algebra which satisfies Theorem~\ref{thm_irreducible_nonrigid_structure}, but which is not a product quotient, since the real Heisenberg algebras are excluded.  
\item If the $\tilde\fg_j$s are copies of the first complex Heisenberg group, then although $K$ does not satisfy the assumptions of Lemma~\ref{lem_converse_product_quotient}.   The quotient  $\tilde\fg/K$ is the second complex Heisenberg algebra, which satisfies Theorem~\ref{thm_irreducible_nonrigid_structure}, and is a product quotient.  Note also that the same graded algebra has an alternate description as $\fh_2^\C/K'$, where $K':=\{0\}$, and in particular $\aut(\tilde\fg/K)$ acts transitively on the first layer.  
\eit

\begin{example}($\Z_5$-product quotients)
\label{ex_z5_product_quotient}
Let $\tilde\fg=\oplus_{j=1}^5\tilde\fg_j$, where $\tilde\fg_j$ is a copy of the first real Heisenberg group with basis $X_{3j-2}, X_{3j-1}, X_{3j}$, $[X_{3j-2},X_{3j-1}]=-X_{3j}$.  Then we have the permuation action $S_5\acts\tilde\fg$, and we restrict this to the subgroup $\Z_5$ which permutes the summands cyclically.  Then the representation $\Z_5\acts\tilde V_2$ is a copy of the permutation representation $\Z_5\acts\R^5$, which is isomorphic to the left regular representation of $\Z_5$.  This decomposes as a direct sum $\tilde V_2=K_1\oplus K_2 \oplus K_3$, where $K_1=\Span(\tilde X_3+\ldots+\tilde X_{15})$ is the diagonal, and 
$$K_2=\Span\{-aX_3+aX_6+X_9-X_{15},  -aX_3-X_6+X_{12}+aX_{15}\},$$
$$K_3=\Span\{X_3-X_6+aX_9-aX_{15},  
   -aX_3+X_6-X_9+aX_{12}\},$$
    with  $a=(\sqrt 5-1)/2$,   
are $2$-dimensional irreducible representations, where the generator of $\Z_5$ acts as rotation by  $\frac{2\pi}{5}$ and $\frac{4\pi}{5}$, respectively.   Then $K_i$ satisfies the assumptions of Lemma~\ref{lem_converse_product_quotient}.  
\end{example}

\medskip
\begin{example}(Permutation product quotients) We may generalize the two preceding examples as follows.  Let $\tilde\fg=\oplus_{j=1}^n\fg_j$, where $n\geq 2$ and $\fg_j$ is a copy of the $m$-th real or complex Heisenberg algebra.   Then permutation of the copies of $\tilde\fg_j$ gives an action $S_n\acts\tilde\fg$, and we may seek an $S_n$-invariant $\R$-linear subspace $K\subset \tilde V_2$ which satisfies \ref{eqn_k_cap_fg_k} and \ref{distinct_2nd_layers}.
\end{example}

\begin{example}
Let $\tilde\fg=\tilde\fg_1\oplus\tilde\fg_2$ where $\tilde\fg_i$ is a copy of the complex Heisenberg graded algebra $\fh_m^\C$.  Let $\tilde V_{2,i}$ be the second layer of $\tilde\fg_i$, so $\tilde V_2=\tilde V_{2,1}\oplus \tilde V_{2,2}$.  Letting $\phi:\tilde V_{2,1}\ra \tilde V_{2,2}$ be an $\R$-linear isomorphism, and $K\subset \tilde V_2=\tilde V_{2,1}\oplus \tilde V_{2,2}$ be its graph, we obtain an example satisfying Lemma~\ref{lem_converse_product_quotient}, see \cite{xie_classification_class_nonrigid} for details.
\end{example}

\begin{example}(Decomposable product quotients)
\label{ex_decomposable_product_quotients}
Suppose $\fg=\tilde\fg/K$ is a product quotient, and let $\fg'$ be   the direct sum of $N\geq 2$ copies of $\fg$.  Then $\fg'$ is also a product quotient, as the conditions of Lemma~\ref{lem_converse_product_quotient} clearly carry over to $\fg'$.   We shall show in the next subsection that every product quotient
can be written canonically as a direct sum of $N \ge 1$ so called conformal product quotients (see Definition~\ref{de:conformally_compact}).
\end{example}

\bigskip

\subsection{Decomposition of product quotients}\label{subsec_decomposition_conformally_compact}
The goal of this subsection is Lemma~\ref{lem_decomposition_conformally_compact}, which asserts that every product quotient admits a canonical decomposition as a  direct sum of indecomposable summands which are conformal, in the sense of Definition~\ref{de:conformally_compact} below.

Let  $(\fg,\{V_j\}_{j=1}^2)$ be a product quotient as discussed in the preceding subsection.  Hence there are canonically defined graded subalgebras $\fg_1,\ldots,\fg_n\subset\fg$, and a graded epimorphism $\pi:\tilde\fg:=\oplus_j\fg_j\ra\fg$, where $\tilde\fg$ has the direct sum grading, and the kernel $K\subset\tilde V_2$ satisfies   \ref{eqn_k_cap_fg_k}--
\ref{eqn_stabilizer_acts_transitively} above.  We have canonical actions $$\aut(\fg)=\stab(K,\aut(\tilde\fg))\acts\tilde \fg$$ and $\aut(\fg)\acts\{1,\ldots,n\}$.

\begin{definition}  \label{de:conformally_compact}
A product quotient $(\fg,\{V_j\}_{j=1}^2)$ is {\bf conformal}  if the action $\aut(\fg)\acts \tilde\fg$ preserves a conformal structure on $\tilde V_2$, i.e. if there is an inner product on $\tilde V_2$ that is $\aut(\fg)$-invariant up to scale.
\end{definition}

\begin{remark}
We emphasize to the reader that graded automorphisms of conformal product quotients act conformally on the second layer, but not necessarily on the first layer. 
\end{remark}

\begin{lemma}
\label{lem_decomposition_conformally_compact}
Every product quotient $\fg$ has graded direct sum decomposition $\fg=\oplus_k\hat\fg_k$, where the summands $\hat\fg_k$ are   conformal product quotients, and the decomposition is respected by $\aut(\fg)$, i.e. each $\Phi\in \aut(\fg)$ induces a permutation of the summands.  Moreover, this decomposition is unique.
\end{lemma}

The proof of the lemma will occupy the remainder of this subsection. We now fix a product quotient $(\fg,\{V_j\}_{j=1}^2)$.  

Our approach to proving Lemma~\ref{lem_decomposition_conformally_compact} is to find a partition of $\{1,\ldots,n\}$ that is compatible with $K$ and the action $\aut(\fg)\acts\{1,\ldots,n\}$.

\begin{definition}
A partition $\{1,\ldots,n\}=J_1\sqcup\ldots\sqcup J_\ell$ is {\bf $K$-compatible} if the direct sum decomposition 
\begin{equation}
\label{eqn_hat_v_2_decomposition}
\tilde V_2=(\oplus_{j\in J_1}\tilde V_{2,j})\oplus\ldots\oplus(\oplus_{j\in J_\ell}\tilde V_{2,j})\,
\end{equation}  
is compatible with $K$:
$$
K=(K\cap\oplus_{j\in J_1}\tilde V_{2,j})\oplus\ldots\oplus(K\cap\oplus_{j\in J_\ell}\tilde V_{2,j}).
$$
\end{definition}

\begin{lemma}

\mbox{}
\ben
\item A partition $\{1,\ldots,n\}=J_1\sqcup\ldots\sqcup J_\ell$ is $K$-compatible iff $K$ is invariant under each projection $\pi_{J_k}:\tilde V_2\ra \oplus_{j\in J_k}V_{2,j}$.  
\item A $K$-compatible partition gives rise to graded decomposition of $\fg$.  If the partition is invariant under the action $\aut(\fg)\acts \{1,\ldots,n\}$, then the resulting graded decomposition of $\fg$ is $\aut(\fg)$-invariant.
\item There exists a unique finest $K$-compatible partition.
\item The partition in (3) induces an $\aut(\fg)$-invariant graded  decomposition of $\fg$ into product quotients.
\een
\end{lemma}
\begin{proof}
(1) and (2) are straightforward.

(3).  If $J_1\sqcup\ldots\sqcup J_\ell$ and $J_1'\sqcup\ldots\sqcup J_{\ell'}'$ are both $K$-compatible then (1) implies that $K$ is invariant under $\pi_{J_k\cap J_{k'}'}=\pi_{J_k}\circ \pi_{J_{k'}'}$ for all $k\in \{1,\ldots,\ell\}$, $k'\in\{1,\ldots,\ell'\}$.  Hence the common refinement $J_1''\sqcup\ldots\sqcup J_{\ell''}''$ is also $K$-compatible by (1).  This implies (3). 

(4).  This follows readily from the definitions.
\end{proof}

\bigskip
We now let $\{1,\ldots,n\}=J_1\sqcup\ldots\sqcup J_\ell$ be the finest $K$-compatible partition given by the lemma.  Letting  
\begin{align*}
\hat V_k:=\oplus_{j\in J_k}\tilde V_{2,j}\,,\qquad \hat K_k:=K\cap \hat V_k\,,\\
 \hat\fg_k:=(\oplus_{j\in J_k}V_{1,j})\oplus (\hat V_k/\hat K_k)\subset\fg\,,
\end{align*} 
we have $\aut(\fg)$-invariant decompositions
$$
\tilde V_2=\oplus_k\hat V_k\,,\qquad K=\oplus_k\hat K_k\,,\qquad \fg=\oplus_k\hat \fg_k\,.
$$

We consider the action $\aut(\fg)\acts \tilde\fg$.  Let $D\subset \aut(\fg)$ be the subgroup that leaves each summand $\fg_j\subset\tilde\fg$ invariant, and acts on each summand by $\F$-linear automorphisms.  Then $D$ has finite index in $\aut(\fg)$ because by Lemma~\ref{lem_complex_heisenberg_auts} the $\F$-linear automorphisms of $\fg_j$ have index at most $2$ in its full automorphism group (as a graded algebra over $\R$).  Thus $D$ acts on $\tilde V_2=\oplus_j\tilde V_{2,j}\simeq \F^n$ by linear transformations that are diagonalizable over $\F$, and even when $\F=\C$, these admit generalized eigenspace decompositions over $\R$ (i.e. direct sum decompositions into $2$-dimensional irreducible subspaces).  

\begin{lemma}
\label{lem_eigenvalues_same_modulus}
If $A\in D$ and $k\in \{1,\ldots,\ell\}$, then the restriction of $A$ to $\hat V_k$ has a single eigenvalue if $\F=\R$ or eigenvalues with the same modulus if $\F=\C$. 
\end{lemma}
\begin{proof}
Suppose $A\restr_{\hat V_k}$ has distinct eigenvalues if $\F=\R$, or two eigenvalues with distinct modulus, if $\F=\C$.  Then the (generalized) eigenspaces of $A\restr_{\hat V_k}$ on $\hat V_k$ give a nontrivial direct sum decomposition of $\hat V_k$ corresponding to a nontrivial partition  $J_k=\bar J_1\sqcup\ldots\sqcup \bar J_{\bar\ell}$:
\begin{equation}
\label{eqn_hat_v_k_decomposition}
\hat V_k=(\oplus_{j\in \bar J_1}\tilde V_{2,j})\oplus\ldots\oplus(\oplus_{j\in \bar J_\ell}\tilde V_{2,j})\,.
\end{equation}
Notice that $\id_{\tilde V_k}=\pi_1+\ldots+\pi_{\bar\ell}$ where the $\pi_i$s are idempotents realizing the decomposition (\ref{eqn_hat_v_k_decomposition}), and each $\pi_j$ is a polynomial in $A$ with coefficients in $\R$.     Since $K$ is $D$-invariant, we therefore obtain a nontrivial decomposition
$$
\hat K_k=(\hat K_k\cap\oplus_{j\in \bar J_1}\tilde V_{2,j})\oplus\ldots\oplus(\hat K_k\cap\oplus_{j\in \bar J_{\bar\ell}}\tilde V_{2,j})\,.
$$  
This contradicts the assumption that $J_1\sqcup\ldots\sqcup J_\ell$ is the finest $K$-compatible decomposition of $\tilde V_2$.  
\end{proof}  

\bigskip
Pick $k\in \{1,\ldots,\ell\}$. Choose an inner product $\langle\cdot,\cdot\rangle_k$ on $\hat V_k$ such that:
\begin{itemize}
\item The subspaces $\{V_{2,j}\}_{j\in J_k}$ are $\langle\cdot,\cdot\rangle_k$-orthogonal.
\item If $\F=\C$, then complex multiplication is $\langle\cdot,\cdot\rangle_k$-orthogonal.
\end{itemize}
In view of the definition of $D$ and Lemma~\ref{lem_eigenvalues_same_modulus},  $D$ acts conformally on $\langle\cdot,\cdot\rangle_k$.  Since $D$ has finite index in the stabilizer $\stab(J_k,\aut(\fg))$,  there is an inner product 
$\langle\cdot,\cdot\rangle_k'$ on $\hat V_k$ on which $\stab(J_k,\aut(\fg))$ acts conformally. 

We now turn to the uniqueness assertion.  

We will show that the summands of any decomposition as in the statement of Lemma~\ref{lem_decomposition_conformally_compact} are nonabelian and indecomposable; it then follows from \cite[Lemma 7.3]{KMX1}  that the decomposition is unique.   The  invariance under $\aut(\mathfrak g)$ follows from uniqueness. 

We claim that for every $k$, the summand $\hat\fg_k$ cannot be decomposed  nontrivially as a direct sum  of graded algebras.  Arguing indirectly, let $\hat\fg_k=\fh_1\oplus\fh_2$ be a nontrivial decomposition of $\hat\fg_k$ as a direct sum.  The $2$-parameter family of graded automorphisms arising from independent scalings of the summands will contradict the fact that $\hat\fg_k$ is conformal, unless one of the two summands $\fh_i$ has trivial second layer.  However, this would mean that $\fh_i$ is abelian, and moreover is an ideal of $\fg$ contained in the first layer.  This is impossible, since every nonzero element of the first layer of $\fg$ has rank $\geq 1$.   This proves the claim.

This completes the proof of Lemma~\ref{lem_decomposition_conformally_compact}.

\section{Rigidity of product-quotients}
\label{sec_rigidity_product_quotients}
The remainder of the paper is concerned with the rigidity of mappings $G\supset U\ra U'\subset G$ where $G$ is a product quotient -- one of the nonrigid Carnot groups that emerged from the structure theorem in Section~\ref{sec_structure_nonrigid_irreducible_first_layer}.  In this section we will state the main rigidity result for such groups (Theorem~\ref{thm_rigidity_product_quotient}),  
  and give the proof using results from Sections~\ref{sec_dim_k_equals_1}-\ref{sec_higher_product_quotients}.

For the remainder of the paper we let $G=\tilde G/\exp(K)$ be a product quotient (Definition~\ref{def_product_quotient}), and will use the notation from Lemma~\ref{lem_converse_product_quotient} for  $\tilde \fg=\oplus_{i=1}^n\tilde\fg_i$, $K\subset\tilde V_2$, etc.

\begin{theorem}
\label{thm_rigidity_product_quotient}
Let $G=\tilde G/\exp(K)$ be a product quotient of homogeneous dimension $\nu$, and $\fg= V_1\oplus  V_2$ be the layer decomposition of its graded Lie algebra.  Suppose $p>\nu$, $U\subset G$ is open, and $f:G\supset U\ra G$ is a $W^{1,p}_{\loc}$-mapping  such that the sign of the determinant of $D_Pf(x):\fg\ra\fg$ is  constant almost everywhere.    Then there is a locally constant assignment of a permutation $U\ni x\mapsto \si_x\in S_n$ such that:
\ben
\item For a.e. $x\in U$, the Pansu derivative $D_Pf(x)$ permutes the subalgebras $\fg_1,\ldots,\fg_n\subset\fg$ in accordance with the permutation $\si$:
$$
D_Pf(x)(\fg_i)=\fg_{\si_x(i)}\,.
$$
\item For every $x\in U$, cosets of $G_i$ are mapped to cosets of $G_{\si_x(i)}$ locally near $x$:  there is neighborhood $V_x$ of $x$ such that for every $1\leq j\leq n$, and every $y\in V$, 
$$
f(V_x\cap yG_i)\subset f(y)G_{\si_x(i)}\,.
$$
\een 
\end{theorem}
\begin{proof}
Note that (2) follows from (1) by Lemma~\ref{lem_preservation_cosets}.

We may assume that $n\geq 2$ since otherwise the theorem is trivial.

{\em Case 1. Each $\tilde \fg_i$ is either a copy of a complex Heisenberg algebra or a real Heisenberg algebra $\fh_m$ with $m\geq 2$.} This case is handled in Section~\ref{sec_higher_product_quotients}.

{\em Case 2. Each $\tilde\fg_i$ is a copy of the first real Heisenberg algebra for all $i$.}

{\em Case 2a. $\fg$ does not have a nontrivial decomposition as a  direct sum of graded Lie algebras.}  Then $\dim K\geq 1$, and the cases when $\dim K =1 $ and $\dim K\geq 2$ are treated in Sections~\ref{sec_dim_k_equals_1} and \ref{sec_conformal_case}, respectively. 

{\em Case 2b. $\fg$ has a nontrivial direct sum decomposition $\fg=\oplus_j\hat\fg_j$ where the   $\hat\fg_j$s  are indecomposable.} By \cite[Theorem 1.1]{KMX1}, after shrinking the domain and postcomposing with an isometry $G\ra G$ which permutes the factors, we may assume that $f$ is a product mapping.  Applying Case 2a to each factor mapping, the theorem follows. 
\end{proof}

\bigskip
We remark that the proofs in Cases 1, 2a, and 2b are all ultimately based on an application of the Pullback Theorem (Theorem~\ref{co:pull_back2}),  
$$ \int_U  f_P^*(\alpha) \wedge d(\varphi \beta) = 0,$$
with $\varphi \in C_c^\infty(U)$ and suitable closed left invariant forms $\alpha$ and $\beta$  which satisfy the conditions
 $\deg \alpha + \deg \beta = N -1$
and $\wt(\alpha) + \wt(\beta) \le - \nu +1$.   
A more detailed motivation  for the choice of the forms $\alpha$ and $\beta$ is given at the beginning of 
Sections~\ref{sec_dim_k_equals_1}-\ref{sec_higher_product_quotients}  where the three different cases are treated. 
Here we just outline some guiding principles which helped us to identify good classes of left invariant forms.
\begin{enumerate}[label=(\roman*)]
\item The pullback of $\alpha$ by the Pansu differential $D_P f(x)$ (or its lift to $D_P \tilde f(x)  \in \stab(K, \aut{\tilde \fg})$, see Lemma~\ref{lem_autg_lifts_to_auttildeg}) should
detect the permutation $\sigma_x$ without depending on  too detailed information on $D_P f$;
\item We should work with closed forms  $\alpha$ modulo exact forms (see \cite[Lemma 4.8]{KMX1});
\item There should be a sufficient supply of ``complementary'' forms $\beta$ which are closed but not exact.
\een
Regarding  (i)  we found it useful to look for forms which are related to the action of $D_Pf$ on the second layer $V_2$
(or its lift $\tilde V_2$)
as this action captures the permutation but is in general much simpler.
This would suggest to take $\alpha$ as linear combination of  wedge products of forms in $\Lambda^1  V_2$.
Such forms are, however, in general not closed:  the exterior derivative maps $\Lambda^1 V_2$ to $\Lambda^2 V_1$ and in fact to 
$\oplus_{i=1}^n \Lambda^2 V_{1,i}$.  Thus  we consider linear combinations of wedge products of  $\Lambda^2 V_{1,i}$ and $\Lambda^1 V_2$, and look for closed forms modulo exact forms.

\section{Product-quotients with $\dim K=1$}
\label{sec_dim_k_equals_1}
In this section and the next we prove Theorem~\ref{thm_rigidity_product_quotient} for product quotients $G=\tilde G/\exp(K)$ where $\tilde G$ is a product of copies of the first Heisenberg group;  we remark that in this case it suffices to assume that the Pansu differential is an isomorphism almost everywhere -- the sign condition on the determinant will not be used. In Subsection~\ref{subsec_notation_products_first_heisenberg} we fix notation that will remain in force for this section and the next.  In Subsection~\ref{subsec_reduction_to_diagonal} we show that the case when $\dim K=1$ reduces to the  case that $K$  is the diagonal subspace $\sum_m\tilde Y_m$.      In Subsection~\ref{subsec_local_constancy_permutation_dim_k_equals_1} we carry out the proof of the rigidity theorem, after first  motivating the argument by comparing with the product case.

\bigskip
\subsection{Notation for product quotients of the first Heisenberg group}
\label{subsec_notation_products_first_heisenberg}
\mbox{}  Let $\fh$ denote the first real Heisenberg algebra. We denote by $\tilde X_1$, $\tilde X_2$, $\tilde X_3$
its standard basis with $[\tilde X_1, \tilde X_2] = - \tilde X_3$. As usual we identify
elements of the algebra with left invariant vector fields of the first real Heisenberg group $H$.
The elements of dual basis are denoted by
$\tilde \theta_i$ and satisfy $d \tilde \theta_3 = \tilde \theta_1 \wedge \tilde \theta_2$. 

Let $\tilde \fg = \oplus_{i=1}^n \tilde\fg_i$ where $\tilde\fg_i$ is a copy of $\fh$, and let 
$\tilde X_1, \ldots \tilde X_{3n}$ be the basis of $\tilde \fg$, where $\tilde X_{3i-2},\tilde X_{3i-1},\tilde X_{3i}$     form  a standard basis for $\tilde \fg_i$, i.e. 
$[\tilde X_{3i-2}, \tilde X_{3i-1}] = - \tilde X_{3i}$. The dual basis is $\tilde \theta_{1}, \ldots \tilde \theta_{3i}$
with $d \tilde \theta_{3i} = \tilde \theta_{3i-2} \wedge \tilde \theta_{3i-1}$. 
The first layer subspace $\tilde V_{1,i}$  is spanned by $\tilde X_{3i-2}$ and $\tilde X_{3i-1}$ and the second layer
subspace $\tilde V_{2,i}$  is spanned by $\tilde X_{3i}$. 
To shorten notation,  we will identify $\tilde V_{1,j}$ with $V_{1,j}$, and drop the tilde when denoting first layer vectors 
and forms. We further introduce the shorthand notation
\begin{equation} \label{eq:define_tau_Y}
 \tilde \tau_i = \tilde \theta_{3i},  \quad \tilde Y_i = \tilde X_{3i}
 \end{equation}
and
\begin{equation}  \label{eq:define_gamma_Z}  \gamma_i = \theta_{3i-2} \wedge  \theta_{3i-1}, 
\quad Z_i = X_{3i-2} \wedge  X_{3i-1}.
\end{equation}
We also consider 
\begin{equation} \label{eq:define_tau}
 \hbox{$\tau$, \quad a volume form  on $V_2 = \tilde V_2 /K$,}
 \end{equation}
and 
\begin{equation}  \label{eq:define_omega}
 \hbox{$\omega = \gamma_1 \wedge \ldots \wedge \gamma_n \wedge \tau$, \quad a volume form on $\fg$.}
 \end{equation}

\subsection{Reduction to convenient $K$ by anisotropic dilation}
\label{subsec_reduction_to_diagonal}
\mbox{} For the remainder of this section, we will assume that $\dim K=1$. 

We now show that without loss of generality we may take $K=\Span(\sum_{i=1}^n\tilde Y_i)$.

\begin{lemma}
\label{eq:dim1_diagonal}
There exists a graded automorphism $\Psi: \tilde\fg\ra\tilde\fg$ such that $\Psi(\Span(\sum_{i=1}^n\tilde Y_i))=K$.  In particular, $\Psi$ induces an isomorphism of graded algebras $\tilde\fg/(\sum_i\tilde Y_i)\ra  \tilde\fg/K$.  
\end{lemma}
\begin{proof}
Suppose that $K \subset \oplus_{i \ne i'} V_{2,i}$ for some $i'$. 
Since $\stab(K, \aut(\tilde \fg))$ acts transitively on the collection of subspaces $\{V_{1,i}\}_{1\leq i\leq n}$    and hence also on the collection $\{V_{2,i}\}_{1\leq i\leq n}$
 we deduce that $K \subset \oplus_{i \ne i'} V_{2,i}$ for all $i'$, so $K\subset \cap_{i'}(\oplus_{i \ne i'} V_{2,i})=\{0\}$. This contradicts the fact that $\dim K=1$.

Therefore we have $K=\Span(\sum_i\mu_i\tilde Y_i)$ where $\mu_i\neq 0$ for all $i$.  

For every $i$, let $\Psi_i:\fh\simeq \tilde\fg_i\ra\tilde \fg_i$ be a graded automorphism with $\Psi_i\mid_{\tilde V_{2,i}}=\mu_i\id_{\tilde V_{2,i}}$, and define $\Psi:=\oplus\Psi_i:\tilde\fg\ra \tilde\fg$.  Then $\Psi(\sum_i\tilde Y_i)=\sum_i\mu_i\tilde Y_i$ as desired.
\end{proof}

In view of Lemma~\ref{eq:dim1_diagonal} we will from now on assume without loss of generality that
\begin{equation}  \label{eq:K_diagonal}
 K =  \Span( \sum_{i=1}^n \tilde Y_i).
\end{equation}

\begin{proposition}  \label{le:auto_preserve_diagonal}

If $\Phi\in\aut(\fg)$ and $\tilde\Phi \in \stab(K, \aut(\tilde \fg))$ is the unique lift provided by Lemma~\ref{lem_autg_lifts_to_auttildeg}, then there exist
a $\lambda \not= 0$ and a permutation $\sigma \in S_n$ such that
\begin{equation}  \label{eq:pull_back_dim_1}
\tilde\Phi^*\tilde  \tau_i = \lambda \tilde \tau_{\sigma^{-1}(i)},\quad \tilde\Phi^*\tilde\gamma_i=\lambda\tilde\gamma_{\sigma^{-1}(i)}, \quad \Phi^* \gamma_i = \lambda \gamma_{\sigma^{-1}(i)}
\quad \forall i =1, \ldots, n,  
\end{equation}
or, equivalently, 
\begin{equation}  \tilde\Phi_*\tilde  Y_i = \lambda \tilde Y_{\sigma(i)}, \quad \tilde\Phi_* (\tilde Z_i) =
\lambda \tilde Z_{\sigma(i)},\quad \Phi_* (Z_i) =
\lambda Z_{\sigma(i)}
\quad \forall i =1, \ldots, n
\end{equation}
where $Z_i = X_{3i-2} \wedge X_{3i-1}$. 
\end{proposition}

\begin{proof} Since $\Phi$ permutes the collections $\{\tilde V_{1,i}\}$, $\{\tilde V_{2,i}\}$  there exist
$\mu_i$ such that
$$ \Phi \tilde Y_i = \mu_i \tilde Y_{\si(i)}.$$
Since $\Phi$ preserves $K$ the vector $\sum_{j=1}^n \tilde Y_j$ must be an eigenvector of $\Phi$
$$ \Phi \sum_{j=1}^n \tilde Y_j = \lambda \sum_{j=1}^n \tilde Y_j.$$
Comparison gives $\mu_i = \lambda$ for all $i=1, \ldots, n$. 
From this identity all the assertions follow easily. 
\end{proof}

\bigskip
\subsection{Local constancy of the permutation}
\label{subsec_local_constancy_permutation_dim_k_equals_1}

\mbox{}
Let $f:G\supset U\ra G$ be a $W^{1,p}_{\loc}$-mapping for some $p>\nu$, and assume that the Pansu differential $D_Pf(x):\fg\ra\fg$ is an isomorphism for a.e. $x\in U$.

Our approach to controlling $f$ is motivated by the treatment of products $\prod_j\H$, where we applied the pullback theorem to the wedge product 
$ f_P^*(d \tilde \tau_j \wedge \tilde \tau_j)\wedge d(\varphi\hat \be)$ 
where $1\leq j\leq n$ and $\hat\be$ is a suitably chosen closed codegree $4$ form.  The form $\hat\be$ may be expressed as $i_X\be$ where 
$$
\be=i_{Y_m} i_{X_{3m-2}} i_{X_{3m-1}} \omega=\pm d\tilde\tau_1\wedge\tilde\tau_1\wedge\ldots\widehat{d\tilde\tau_m\wedge\tilde\tau_m}\wedge\ldots\wedge d\tilde\tau_n\wedge\tilde\tau_n
$$
for some  $m \in \{1, \ldots, n\}$ and $X\in \oplus_{j\neq m}V_{1,j}$.

Note that in the present context the forms $\tilde\tau_i$ are not well-defined because they do not descend to the quotient $\fg=\tilde\fg/K$.  We observe that the differences $\tilde\tau_{i,j}:=\tilde\tau_i-\tilde\tau_j$ descend to well-defined forms $\tau_{i,j}$ for all $i\neq j$.  This motivates our choice of the  $3$-forms
\begin{equation}  \label{eq:define_omega_ij}
\om_{ij}:=(\ga_i+\ga_j)\wedge (\tau_{i,j})
\end{equation}
whose pullback to $\tilde G$ is $(d\tilde \tau_i  + d\tilde\tau_j) \wedge(\tilde\tau_i -\tilde\tau_j)$ for $1\leq i\neq j\leq n$.  We will apply the pullback theorem to $f_P^*\om_{ij}\wedge d(\varphi \,i_X\beta)$ where $\beta$ is a closed codegree $3$ form given below.
It follows from Proposition~\ref{le:auto_preserve_diagonal} that
\begin{equation}  \label{eq:pullback_dim_K_1}
 f_P^*(\omega_{ij}) = \lambda^2(x) (\omega_{\sigma_x^{-1}(i) \sigma_x^{-1}(j)}).
 \end{equation}

At the heart of the matter are  the following two identities.

\begin{lemma} \label{le:key_identities_dim_K_equals_1}
Let $m \in \{1, \ldots, n\}$ and set 
\begin{equation} \label{eq:define_beta_dim_1}
\beta = i_{Y_m} i_{X_{3m-2}} i_{X_{3m-1}} \omega
\end{equation}
where $Y_m:=\tilde Y_m+K$ is the image of $\tilde Y_m$ under the projection $\tilde\fg\ra \fg$.  Then 
\begin{equation} \label{eq:beta_dim_1_closed}
 d(i_X \beta) = 0   \quad \forall X \in \oplus_{j \ne m}V_{1,j}
\end{equation}
and $i_X \beta$ is a form of codegree $4$ and weight $\le - \nu + 5$.
Moreover, 
\begin{align}  \label{eq:key_wedge_product_diagonal}
& \, \omega_{kl} \wedge d( \varphi i_X \beta) =
 \begin{cases}
0 & \hbox{if $k \ne m$ and $l \ne m$} \\
-(X \varphi) \omega  & \hbox{if $k=m$ and $l \ne m$} \\
 (X \varphi) \omega & \hbox{if $k \ne m$ and $l=m$.}
\end{cases}
\end{align}
\end{lemma}

\begin{proof}
Using (\ref{eq:i_circ_i}),   \eqref{eq:exterior_codegree5},   (\ref{eq:exterior_codegree3}),   
  and the identity
$[X_{3m-2}, X_{3m-1}] = - Y_m$ we get 
\begin{align*}
d(i_X \beta) = & \, - d(i_{Y_m} i_{X_{3m-2}} i_{X_{3m-1}} i_X \omega) 
=   i_{Y_m} d(i_{X_{3m-2}} i_{X_{3m-1}} i_{X} \omega) \\
= & \,- i_{Y_m}  i_X   i_{[X_{3m-2}, X_{3m-1}]} \omega =  i_{Y_m} i_X i_{Y_m} \omega = 0
\end{align*}
since $i_{Y_m} \circ i_{Y_m} = 0$.

Applying (\ref{eq:i_circ_i}) and (\ref{eq:inner_product_volume_form}), we get: 
\begin{align*}
& \,  \omega_{kl} \wedge d( \varphi i_X \beta)\\
 = & \, 
 \omega_{kl} \wedge d\varphi \wedge i_X  i_{Y_m} i_{X_{3m-2}} i_{X_{3m-1}} \omega \\
 = & \,  (\omega_{kl} \wedge d\varphi)(  {X_{3m-1}}, X_{3m-2}, Y_m, X)  \, \omega\\
 = & \, \omega_{kl} (X_{3m-1}, X_{3m-2}, Y_m) \,   d\varphi(X) \, \omega.
\end{align*}
In the last step we used that 
$$
\omega_{kl}(X, X_{3m-2}, X_{3m-1}) = (\ga_k+\ga_l)\wedge\tau_{k,l}(X, X_{3m-2}, X_{3m-1})=0
$$ 
because $ \tau_{k,l}$ vanishes on $V_1$,
and $\omega_{kl}(X, X_{3m-i'}, \cdot) = 0$ for $i' \in \{1, 2\}$  because $\gamma_j(X, X') = 0$ for $X' \in V_{m,1}$ and $X \notin V_{m,1}$. 
Since $\gamma_j$ vanishes on $V_2$  and $\tau_{k,l}$ vanishes on $V_1$  we get
\begin{eqnarray*}
\omega_{kl} (X_{3m-1}, X_{3m-2}, Y_m) &=& 
    (\gamma_k +  \gamma_l)(X_{3m-1}, X_{3m-2})  \,  \tau_{k,l}(Y_m)  \\
&=&
     - (\delta_{km} + \delta_{lm})  \, (\delta_{km} - \delta_{lm})\\
     &=&-(\delta_{km} - \delta_{lm})
     \end{eqnarray*}
This concludes  the proof of  \eqref{eq:key_wedge_product_diagonal}.
\end{proof}

\bigskip
We now choose $X$ and $\beta$ as in Lemma  \ref{le:key_identities_dim_K_equals_1}
and let 
\begin{equation} \label{eq:define_P_mi}
P_{mi}(x) = \delta_{m \sigma_x^{-1}(i)} = \delta_{\sigma_x(m) i}\,.
\end{equation}
Since $i_X \beta$ is closed, has codegree $4$ and  weight $\le - \nu + 5$, while $\omega_{ij}$ is closed, has degree $3$ and weight $-4$ we can apply Theorem~\ref{co:pull_back2}  to
$f_P^*(\omega_{ij}) \wedge d(\varphi i_X \beta)$
and using      \eqref{eq:pullback_dim_K_1}  and \eqref{eq:key_wedge_product_diagonal}
we get
\begin{equation}
\label{eqn_x_directional_derivative}
X[\la^2(P_{mi}-P_{mj})]=0
\end{equation}
in the sense of distributions, for every $i,j,m\in\{1,\ldots,n\}$  and every  $X\in V_{1,\ell}$ with 
$l \ne m$.  

\begin{lemma}   \label{le:constant_permute_K_equals_1}
The permutation $\sigma$ and function $\lambda^2$ are constant almost everywhere.
\end{lemma}

\begin{proof} Let $h: \R \to [0, \infty)$ be given by $h(t) = \max(t,0)$. Then by (\ref{eqn_x_directional_derivative}) and
Lemma~\ref{le:compact_directional_constancy}
 $$X h\big( \lambda^2 (P_{mi} - P_{mj})\big) = 0$$
 for every $i,j,m\in\{1,\ldots,n\}$  and every  $X\in V_{1,\ell}$ with 
$l \ne m$.  
 Now assume $j \ne i$ and note that $P_{mi} - P_{mj} \in \{-1, 0, 1\}$
 and that $P_{mi}  - P_{mj} = 1$ if and only if $P_{mi} = 1$. 
 Thus 
 $$ h\big( \lambda^2 (P_{mi} - P_{mj})\big) = \lambda^2 P_{mi}.$$
 Hence
 $$ \forall i,m \, \, \forall l \ne m \, \, \forall X \in V_{1,l}   \qquad X (\lambda^2 P_{mi})= 0.$$
 Since $\sum_{i=1}^n P_{mi} = 1$ we get
 $$ X \lambda^2 = 0$$
 for all $X \in V_{1,l}$ and $l \ne m$. Choosing $m=1$ and $m=2$ we conclude that $X \lambda^2=0$
 for all $X \in V_1$ and hence that $\lambda^2$ is constant a.e.\ and 
 $$ \forall i,m \, \, \forall l \ne m \, \, \forall X \in V_{1,l}   \qquad X P_{mi}= 0.$$
In particular
$$ X P_{m'i} = 0   \quad \hbox{for $m' \ne m$ and $X \in V_{1,m}$.}$$
Since $P_{mi} = 1- \sum_{m' \ne m} P_{m'i}$ we deduce that $X P_{mi}=0$
for all $X \in V_1$. Thus the matrix $P$ and the permutation $\sigma$ are constant a.e. 
\end{proof}

\section{Indecomposable product-quotients with $\dim K\geq 2$}
\label{sec_conformal_case}
In this section we will retain the notation and conventions from Subsection~\ref{subsec_notation_products_first_heisenberg}.  The goal of this section is to prove Theorem~\ref{thm_rigidity_product_quotient} for indecomposable product quotients $G = (\prod_{i=1}^n G_i)/\exp(K)$  where $G_i$ is a copy of the first Heisenberg group and  $\dim K \ge 2$.      The argument in this case  is significantly more complicated than the case when $\dim K=1$, so we will first give some motivation and an overview in Subsection~\ref{subsec_overview_dim_k_geq_2}.

\subsection{Overview of the proof}
\label{subsec_overview_dim_k_geq_2}

\mbox{}
We recall that in the $\dim K=1$ case, our approach was to apply the pullback theorem to the closed left invariant forms $\om_{ij}$ and $i_X\be$, which are of degree $3$ and codegree $4$, respectively, and have appropriate weight.  It turns out  that when $\dim K\geq 2$ such forms typically do not exist, so we use a different strategy.  

We first reduce to the case when the lift $\tilde\Phi\in\stab(K,\aut(\tilde\fg))$ of every graded automorphism $\Phi\in \aut(\fg)$ acts orthogonally on $\tilde V_2$ with respect to the inner product induced by the basis $\{\tilde Y_1,\ldots,\tilde Y_n\}\subset \tilde V_2$.  In particular letting $\tilde D_Pf(x)$ denote the lift of the Pansu differential $D_Pf(x)$ of our
  Sobolev mapping    $f:U\ra G$, we note that $f$ is somewhat reminiscent of a conformal mapping, in the sense that its differential is conformal, although only when restricted to the second layer.

The above observation suggests that we look to the proof of Liouville's theorem for inspiration. In its simplest form it states that 
a Lipschitz  map $f: U \subset \R^n \to \R^n$ which satisfies $\nabla f \in SO(n)$ a.e. \  has locally constant 
gradient. The key idea is to use the relations $\mathop{\mathrm{curl}} \nabla f = 0$ and
$\Div \cof \nabla f = 0$ for every Lipschitz map. In connection with the pointwise identity 
$\cof \nabla f = \nabla f$ a.e. \  they imply that $f$ is weakly harmonic and hence smooth.  
The differential constraints on $f$ can be  derived 
by pulling back constant differential forms of degree $1$ and codegree $1$. 
In our setting we are interested in the action of the Pansu differential on the second layer. 
This action can be captured by using special degree $2$ forms in the first layer (namely elements of
$\oplus_{i=1}^n\Lambda^2 V_{1,i}$). Thus the analogy with the proof of Liouville's theorem suggests    the use of the closed degree $2$  forms $\gamma_i = \theta_{3i-2} \wedge \theta_{3i-1}$ and closed  codegree  2 
   forms of the type $i_Z \omega$ where
$\omega$ is a volume form on $G$ (see   \eqref{eq:define_omega}) and $Z$ is a two-vector in  $\oplus_{i=1}^n \Lambda_2 V_{1,i}$.  
Indeed Lemma~\ref{le:compact1}  can be seen as a counterpart of the property $\mathop{\mathrm{curl}} \nabla f = 0$
while (\ref{eqn_x_sp_t}) can be viewed   as a counterpart of  the identity $\Div \cof \nabla f = 0$.

The use of these forms will show that the lift $\tilde D_Pf(x)\in \stab(K,\aut(\tilde\fg))$ of the Pansu differential $D_P f(x)$ acts as a constant on $K$
(see Theorem~\ref{th:trivial_action_on_K}   below).
We then look at the group $H'$  of graded  automorphisms of $\tilde \fg = \oplus_{i=1}^n  \tilde\fg_i$ which act as the identity on
$K$.  If the induced action $H'\acts I=\{1,\ldots,n\}$  (which reflects the permutation of the collection $\{V_{2,i}\}$  by such a graded automorphism)    is trivial we are done. 
If not, then we have the induced decomposition $I=I_1\sqcup\ldots I_k$ into $H'$-orbits, 
and it is easy to see that the projection of $K$ onto each of the corresponding subspaces 
$V_{I_{j}} =
 \oplus_{k \in I_{j}}  V_{2, k}$     is one-dimensional (see Lemma~\ref{lem_dim_k_j_equals_1}   below).
 Then we can use the argument for the case $\dim K = 1$ to show that the permutation induced by $D_P f(x)$
 is locally constant on each orbit and hence locally constant (see Section~\ref{subsec_constancy_permutation} below).

\subsection{Set-up}
\label{subsec_setup_dim_k_geq_2}  Retaining the notation from Section~\ref{subsec_notation_products_first_heisenberg} we consider  $\tilde \fg =  \sum_{i=1}^n \tilde\fg_i$ where $\tilde\fg_i$ is a copy of the first Heisenberg algebra, and let $\fg = \tilde \fg / K$ be an indecomposable   product-quotient.  
By  Lemma~\ref{lem_decomposition_conformally_compact} there exists a scalar product $b$ on $\tilde V_2$ such that the action of  $\stab(K, \aut(\tilde \fg))$
is conformal, i.e., invariant up to scaling; in other words, $\fg$ is a conformal product quotient, see Definition~\ref{de:conformally_compact}. 

\begin{proposition}  \label{eq:conformal_standard_scalar_product}
There exists an anisotropic dilation 
  $\delta_{\boldsymbol{\mu}}$  such that 
the action of $\stab(\delta_{\boldsymbol{\mu}}K, \aut(\tilde \fg))$ is conformal with respect to the standard scalar product on $V_2$
which is characterized by the identities
$(\tilde Y_i, \tilde Y_j) = \delta_{ij}$.
\end{proposition}

\begin{proof}
We first note that for every anisotropic dilation   $\delta_{\boldsymbol{\mu}}$ the quotient
$\tilde \fg/ \delta_{  \boldsymbol{\mu}}K$ is an indecomposable product quotient and
that 
 the push-forward scalar product
$(\delta_{\boldsymbol{\mu}})_* b$ is $\aut(\tilde \fg/ \delta_{\boldsymbol{\mu}}K)$ invariant up to scaling.
Indeed,  $\delta_{\boldsymbol{\mu}}$  is a graded 
automorphism of $\tilde g$ that sends the ideal $K$ to $\delta_{\boldsymbol{\mu}}K$. Thus $\delta_{\boldsymbol{\mu}}$
 induces a graded  isomorphism
 from $\tilde g/K$  to $\tilde g/{\delta_{\bf \mu}K}$.  
 Moreover   $\delta_{\boldsymbol{\mu}}: (\tilde V_2, b)\rightarrow (\tilde V_2, (\delta_{\boldsymbol{\mu}})_*b)$ is a linear isometry
 and  $\stab(\delta_{\boldsymbol{\mu}}K, \tilde \fg) = \delta_{\boldsymbol{\mu}}
  \circ \stab(K, \fg) \circ \delta_{\boldsymbol{\mu}}^{-1}$. Hence $\stab(\delta_{\boldsymbol{\mu}}K, \tilde \fg)$ preserves 
 $ (\delta_{\boldsymbol{\mu}})_*b$ up to scaling.

We now choose a special dilation by taking  $\mu_i = (b(Y_i, Y_i))^{1/2}$.
Then
\begin{equation} \label{eq:norm_pushforward}
 ( \delta_{\boldsymbol{\mu}})_* b(\tilde Y_i, \tilde Y_i) =
 b(  \delta_{\boldsymbol{\mu}}^{-1} \tilde Y_i,   \delta_{\boldsymbol{\mu}}^{-1} \tilde Y_i) 
 = \mu_i^{-2} b(\tilde Y_i, \tilde Y_i) = 1.
 \end{equation}

 The action of $\Phi \in \stab(\delta_{\boldsymbol{\mu}} K, \aut(\tilde \fg))$ on the second layer
 can be written as a product of a permutation and an invertible   diagonal operator.
 In view of   \eqref{eq:norm_pushforward}   permutations preserve $ (\delta_{\boldsymbol{\mu}})_* b$.  Since $\Phi$ preserves $(\delta_{\boldsymbol{\mu}})_* b$ up to a scalar factor the diagonal operator must also preserve that scalar product up to a scalar factor. 
It follows  that the norm of the eigenvalues of the diagonal operator is constant. 
Then  the diagonal operator also preserves the standard scalar product up to a scalar factor. Since permutations preserve the standard scalar  product the assertion follows. 
\end{proof}

To simplify the notation we will assume from now on without loss of generality
\begin{align}  \label{eq:conformal_standard_product}
& \hbox{$ \stab(K, \aut(\tilde \fg))$ acts conformally on the second layer } \\
& \hbox{with respect to  the standard scalar product.} \nonumber 
\end{align}

\begin{proposition}  \label{pr:action_phi_K_ge2}
Let $\Phi \in  \stab(K, \aut(\tilde \fg))$ and let $\sigma = \sigma_\Phi$ be the induced permutation of the 
subspaces $V_{1,i}$ (and $V_{2,i}$). Then there exist
 $\lambda_{\Phi} \in (0, \infty)$, a diagonal matrix $S_\Phi$ with entries $s_i \in \{-1, 1\}$,  and a permutation
 matrix $P_\Phi$ with the following properties. If $m \in \R^n$ and
 $$ Y = \sum_{j=1}^n m_j Y_j,  \quad Z = \sum_{j=1}^n m_j Z_j$$
 then 
 $$ \Phi(Y) = \sum_{i=1}^n ( \lambda_\Phi S_\Phi P_\Phi m)_i Y_i, \quad 
 \Phi_*(Z) =  \sum_{i=1}^n ( \lambda_\Phi S_\Phi P_\Phi m)_i Z_i.
$$
Moreover $(P_\Phi)_{ij} = \delta_{i \sigma_\Phi(j)}$ 
and $S_\Phi$ and $\lambda_\Phi$ are uniquely determined
and 
\begin{equation}  \label{eq:S_P_orthorgonal}
S_\Phi \in O(n), \quad P_\Phi \in O(n).
\end{equation}
(Note that this definition of $P$ differs from the one in Sections~\ref{sec_higher_product_quotients} and \ref{sec_dim_k_equals_1} by a transpose.)
\end{proposition}

\begin{proof} There exist $\mu_i \ne 0$ such that $\Phi(Y_j) = \mu_{\sigma(j)} Y_{\sigma(j)}$. 
It follows from   \eqref{eq:conformal_standard_product} that there exist a $\lambda \in (0,\infty)$ such that
$|\mu_{\sigma(j)}| = \lambda$ for all $j =1, \ldots, n$. Set $s_i = \mu_i/ \lambda$. Then $s_i \in \{-1, 1\}$ and 
$\Phi(Y_j) = \sum_{i=1}^n (\lambda S P)_{ij} Y_i$. Thus the formula for $\Phi(Y)$ follows by linearity. 

The formula for $\Phi(Z)$ is proved similarly.  Alternatively one can use the following facts. There is a unique linear map $L_{[\, ]}: \Lambda_2 V_1 \to V_2$
such that $L_{[\, ]}(X \wedge Y) = [X, Y]$. This map commutes with every graded homomorphism $\Phi$, 
i.e., $\Phi \circ L_{[ \, ]} = L_{[ \, ]} \circ \Phi_*$. For the first Heisenberg group the restriction of $L_{[ \, ]}$
to $\oplus_{i=1}^n \Lambda_2V_{1,i}$ is a
linear isomorphism onto  $V_2$ and $L_{[\, ]}Z_i = - Y_i$.

Uniqueness of $\lambda = | \Phi(Y_j)|$ and the $s_i$ is
verified easily. The property  \eqref{eq:S_P_orthorgonal} follows directly from the definition of $P$ and $S$. 
\end{proof}

\bigskip

 \subsection{Restrictions on the Pansu differential arising from the pullback of  $2$-forms}
 \label{se:restrictions_pansu_differentiable}

\mbox{}
 
For the remainder of Section~\ref{sec_conformal_case} we fix  for some $p>\nu$ a $W^{1,p}_{\loc}$-mapping $f:G\supset U\ra G$, where $U$ is open and the sign of the determinant of the Pansu differential $D_Pf(x):\fg\ra\fg$ is constant almost everywhere.  
 At each point $x\in U$ of Pansu differentiability the Pansu differential, viewed as an automorphism of the algebra
 $\fg = \tilde \fg/ K$,  can be identified with an element of $ \stab(K, \aut(\tilde \fg))$, which we denote by   
 \begin{equation}
 \tilde D_Pf(x) \in  \stab(K, \aut(\tilde \fg)).
 \end{equation}

 Recalling the notation in Proposition~\ref{pr:action_phi_K_ge2} we set
 \begin{equation}
 \lambda(x) = \lambda_{\tilde D_P f(x)}, \quad 
  S(x) = S_{\tilde D_P f(x)}, \quad
   P(x) = P_{\tilde D_P f(x)}.
 \end{equation}

We also define
\begin{equation}
\hat K := \{ m \in \R^n : \sum_{i=1}^n m_i Y_i \in K\}.
\end{equation}

Note that $\tilde D_P f(x)$ preserves $K$ and thus by Proposition~\ref{pr:action_phi_K_ge2}
\begin{equation}
S(x)P(x)  \hat K = \hat K.
\end{equation}

\bigskip
The key calculation is contained in the following result.  Recall that $\ga_i=\th_{2i-1}\we\th_{2i}$ and $Z_i=X_{3i-2}\we X_{3i-1}$.

\begin{lemma} \label{le:key_calculation_degree2_K2}
 Let $k \in \{1, \ldots, n\}$,  $X \in V_{1,k}$
and 
$Z = \sum_{i=1}^n m_i Z_i$. Then 
\begin{eqnarray}
d(i_Z \omega) &=& 0 \quad \forall m \in \hat K,   \label{eq:dimK_ge2_codegree2} \\
d(i_X i_Z \omega) &=& 0 \quad \forall m \in \hat K \cap e_k^\perp
  \label{eq:dimK_ge2_codegree3}.
\end{eqnarray}
For $\varphi \in C_c^\infty(U)$ we have
\begin{equation}  \label{eq:pull_back_dimK2_degree2}
f_P^*(\gamma_i)  \wedge  d (\varphi i_X i_Z \omega) = (\lambda SPm)_i  X \varphi  \, \omega
\quad \forall m \in \hat K \cap e_k^\perp.
\end{equation}
\end{lemma}

\begin{proof}

Since $Z_i = X_{3i-2} \wedge X_{3i-1}$ and $[X_{3i-2}, X_{3i-1}] = - \tilde Y_i$  
we get from   \eqref{eq:i_circ_i} and   \eqref{eq:exterior_codegree2}
$$d ( i_{Z_i} \omega) = d( i_{X_{3i-1}} i_{X_{3i-2}} \omega) = i_{Y_i} \omega$$
where $i_{Y_i} \omega$ is a shorthand for $i_{Y_i + K} \omega$
(see the end of Section~\ref{se:interior_products}). 
If $m \in \hat K$ then $Y := \sum_{i=1}^n m_i Y_i \in K$ and thus
$$ d( i_Z \omega) = i_Y \omega = 0$$
since $\omega$ annihilates $K$.  
Similarly to prove   \eqref{eq:dimK_ge2_codegree3} we use \eqref{eq:exterior_codegree3}  and get
$$ d (i_X i_{Z_i}  \omega) = -   i_{X} i_{Y_i} \omega   \quad \forall i \ne k.$$
Multiplying by $m_i$, summing over $i$ and using that  $m_k=0$ we get
$d(i_X i_Z \omega) = - i_X i_Y \omega = 0$ since $Y \in K$.

To show   \eqref{eq:pull_back_dimK2_degree2} we compute
\begin{align*}
& \,  f_P^*(\gamma_i) \wedge d(\varphi \, i_X i_Z \omega) = 
 f_P^*(\gamma_i) \wedge d\varphi \wedge i_X i_Z \omega \\
= & \, (f_P^*(\gamma_i) \wedge d\varphi)(Z \wedge X) \,  \omega 
=  f_P^*(\gamma_i)(Z)  \, d\varphi(X) \, \omega. \\
\end{align*}
To get the second equality we used the identity $i_X \circ  i_Z = i_{Z \wedge X}$
(see  \eqref{eq:i_circ_i}) and  \eqref{eq:inner_product_volume_form}.
For the last equality we used $m_k = 0$
and $f_P^*(\gamma_i)(X,Y) = 0$ if $X \in V_{1,k}$ and $Y \in V_{1,j}$ with $j \ne k$
since $f_P^*(\gamma_i)$ is a linear combination of the forms $\gamma_j$. 
Now it follows from Proposition~\ref{pr:action_phi_K_ge2} that 
$$ f_P^*(\gamma_i)(Z) = \gamma_i ( (D_Pf)_* Z) = \sum_{j=1}^m (\lambda S P)_{ij} m_j = (\lambda SP m)_i.$$
This concludes the proof of  \eqref{eq:pull_back_dimK2_degree2}.
\end{proof}

\bigskip
 
  \begin{lemma} \label{le:compact1} Let $k \in \{1, \ldots, n\}$.
  Then 
  \begin{equation}  \label{eq:dimK_ge2_distributional_degree2}
  X  (\lambda S P m)    = 0   \quad \forall X \in V_{1,k},  \, \, m \in \hat K \cap e_k^\perp,
  \end{equation}
  in the sense of distributions on $U$
  and 
  \begin{equation}  \label{eq:dimK_ge2_lambda_const}
  \hbox{$\lambda$ is  locally constant almost everywhere in $U$.}
  \end{equation}
 \end{lemma}

\begin{proof} By Lemma~\ref{le:key_calculation_degree2_K2}, the form $i_Xi_Z \om$ is a closed codegree $3$ form with weight $\leq -\nu+3$.  Thus $d(\varphi i_Xi_Z\om)=d\varphi\wedge i_Xi_Z\om$ is a codegree $2$ with weight $\leq -\nu+2$.   Moreover $\ga_i$ is a closed $2$-form of weight $-2$.  Hence the identity   \eqref{eq:dimK_ge2_distributional_degree2}
follows from the Theorem~\ref{co:pull_back2} and equation (\ref{eq:pull_back_dimK2_degree2}).

To prove    \eqref{eq:dimK_ge2_lambda_const} it suffices to show $X \lambda = 0$
for all $X \in V_{1,k}$ and all $k =1, \ldots, n$. Fix $k$. Since $\dim K \ge 2$ the set
$\hat K \cap e_k^\perp$ contains an element $m$ with $|m|=1$. Since $P$ and $S$ are
isometries  and $\lambda > 0$ we have $\lambda =|\lambda SPm|$.
Thus the desired assertion $X \lambda = 0$
follows from Lemma~\ref{le:compact_directional_constancy}  by taking $h(v) = |v|$. 
\end{proof}

\bigskip

\subsection{Restrictions from codegree $2$  forms and constant action on $K$}
\mbox{}
We continue the analysis of the Sobolev map $f:U\ra  G$.
The map $SP$ is an isometry of $\R^n$ and maps $\hat K$ to itself. Hence $(SP)^T = (SP)^{-1}$ also
maps $\hat K$ to itself. Thus (\ref{eq:dimK_ge2_distributional_degree2})  is equivalent to 
\begin{equation}  \label{eq:pull_back_degree2_variant}
 X ((\lambda SP)^T a, m) = 0 \quad \forall a \in \hat K, \, m \in \hat K \cap e_k^\perp, \,  X \in V_{1,k}.
 \end{equation}
We already know that $\lambda$ is locally constant. 
To show that $SP|_{ \hat K}$  is locally constant we need to show that in addition
\begin{equation}  
X ( (SP)^T a, e_k) = 0 \quad \forall a \in \hat K, \,  X \in V_{1,k}.
\end{equation}

This will be achieved by pulling back suitable closed codegree $2$ forms.
The main identities are contained in the following lemma.

\begin{lemma} \label{le:compact_adjugate} Let $\Phi \in \stab(K, \aut(\tilde \fg))$, viewed as an element of  $\aut(\fg)$, 
 let $\lambda_\Phi$, $S_\Phi$
and $P_\Phi$ be the quantities introduced in Proposition~\ref{pr:action_phi_K_ge2}. Let $m \in \hat K$ and
set $Z = \sum_{i=1}^n m_i Z_i$. 
Then 
\begin{eqnarray}   \label{eq:det_Phi_adjugate}
|\det \Phi| &=& \lambda_\Phi^{2n - \dim K},\\
d(i_Z \omega) &=& 0,
\end{eqnarray}
and for every two-form $\alpha$
\begin{equation}  \label{eq:pull_back_codegree2}
\Phi^*(i_Z \omega) \wedge \alpha = \det \Phi \, \alpha( \Phi^{-1} Z) \, \omega.
\end{equation}
\end{lemma}

\begin{proof} The determinant of $\Phi$ is  characterised by the identity
$\Phi^*(\omega) = \det \Phi\,  \omega$. It follows from Proposition~\ref{pr:action_phi_K_ge2}
that 
\begin{equation} \label{eq:det_Phi_first_layer} \Phi^*(\wedge_{i=1}^n \gamma_i)  = \lambda^n \det(SP)  \wedge_{i=1}^n \gamma_i =
\pm \lambda^n \wedge_{i=1}^n \gamma_i.
\end{equation}
Moreover $V_2= \tilde V_2/ K$ can be identified with $K^\perp$ and $\lambda^{-1} \Phi$, viewed as a map on $\tilde V_2$,  is an isometry and preserves $K$ and hence
$K^\perp$. Thus $\lambda^{-1} \Phi$ restricted to $K^\perp$ has determinant $\pm 1$ and hence
$$ \Phi^*(\tau) = \pm \lambda^{n - \dim K} \tau$$
for every volume form $\tau$ on $V_2$. Together with  \eqref{eq:det_Phi_first_layer} this implies 
 \eqref{eq:det_Phi_adjugate}.

The second assertion is just  \eqref{eq:dimK_ge2_codegree2}.

To show  \eqref{eq:pull_back_codegree2} we compute
\begin{align*}
& \,  \Phi^*(i_Z \omega) \wedge \alpha 
=  \Phi^*( i_Z \omega \wedge (\Phi^{-1})^*\alpha) \\
=  & \, \det \Phi  \,  i_Z \omega \wedge (\Phi^{-1})^*\alpha\\
= & \, \det \Phi  \, \big((\Phi^{-1})^*\alpha\big)(Z)  \, \omega 
=  \det \Phi \, \alpha(\Phi^{-1} Z) \, \omega.
\end{align*}
\end{proof}

\begin{theorem}  \label{th:trivial_action_on_K}
The restriction of 
$\lambda S P $ to $\hat K$  is  locally constant almost everywhere in $U$. 
\end{theorem}

\begin{proof} Let $Z$ be as in Lemma~\ref{le:compact_adjugate}, and pick  $k \in \{1, \ldots, n\}$.  Let $\eta = \varphi \theta_{3k-1}$  and $\varphi\in C^\infty_c(U)$.
Thus  $d\eta = d\varphi \wedge \theta_{3k-1}$. Now  $i_Z \omega$ has codegree $2$ and weight $-\nu + 2$ while the form $d\eta$ has weight $\le -2$. Thus the pullback theorem applies to $f_P^*(i_Z \omega) \wedge d\eta$. Since the sign of $\det D_Pf$ is constant almost everywhere by hypothesis, after post-composing with a graded automorphism if necessary, we may assume without loss of generality that    $\det D_P f > 0$ a.e.  and thus 
 Lemma~\ref{le:compact_adjugate} implies that
 $$ f_P^*(i_Z \omega) \wedge d\eta = \lambda^{2n - \dim K} 
 (d\varphi \wedge \theta_{3k-1})((D_Pf)^{-1}(Z)) \, \omega.
 $$
 Since $(D_Pf)^{-1}(Z) \in \oplus_{i=1}^n \Lambda_2 V_{1,i}$   
   only the term with $i=k$ contributes and we get
\begin{align*} 
(d\varphi \wedge& \theta_{3k-1})((D_Pf)^{-1}(Z)) = (X_{3k-2} \varphi)  \, \gamma_k((D_Pf)^{-1} Z)\\
 =& (X_{3k-2} \varphi) \, \lambda^{-1} ((SP)^{-1}m, e_k).
\end{align*}
 Since $\lambda$ is locally constant it follows that
 $$ X_{3k-2} ((SP)^{-1} m, e_k) = 0 \quad \forall m \in \hat K. $$
Using the form $\eta = \varphi \theta_{3k-2}$ we get the same assertion with 
$X_{3k-2}$ replaced by $X_{3k-1}$. Using that $(SP)^{-1} = (SP)^T$ we thus conclude that
 \begin{equation}
 \label{eqn_x_sp_t}  
X ( (SP)^T a, e_k) = 0 \quad \forall a \in \hat K, \,  X \in V_{1,k}.
\end{equation}
Combining this with   \eqref{eq:pull_back_degree2_variant} 
we deduce that $X (SP)^T a = 0$ for all $a \in \hat K$, all $X \in V_{1,k}$ and all $k$. Therefore $\la SP$ is locally constant, as desired. 
\end{proof}

\bigskip
\subsection{Constancy of the permutation}
\label{subsec_constancy_permutation}
\mbox{}
It follows from Theorem \ref{th:trivial_action_on_K} that after composing $f$ with a graded automorphism of $G$ and shrinking $U$ we may assume without loss of generality that
$\tilde D_P f$  
  acts  as the identity on  $K$. In this subsection we analyse the graded automorphisms which act
as the identity on $K$ and then show that the permutation induced by $\tilde D_P f$   is locally constant.

To simplify notation, we let $H:=\aut(\fg)$ be the graded automorphism group of $\fg$.  By Lemma~\ref{lem_autg_lifts_to_auttildeg} we know that $H$ canonically embeds in $\aut(\tilde\fg)$ as the stabilizer of $K\subset\tilde V_2\subset \tilde\fg$.  Therefore we have actions $H\acts\fg$, $H\acts\tilde\fg$ by graded automorphisms, as well as the induced action $H\acts I=\{1,\ldots,n\}$ on the (set of indices of the) summands $\tilde\fg_i$, and the restriction action $H\acts K$.  We let  
$$
H':=\{h\in H\mid h\cdot W=W\,, \;\forall W\in K\}\subset H
$$ 
be the ineffective kernel of the action $H\acts K$; since $H'$ is the kernel of a homomorphism, it is a normal subgroup of $H$.

We let $I=I_1\sqcup\ldots\sqcup I_k$ be the decomposition of $I$ into the distinct orbits of the action $H'\acts I$; as $H'$ is normal in $H$, the $H'$-orbit decomposition is respected by $H$.

We let $\tilde V_{2,I_j}=\oplus_{i\in I_j}\tilde V_{2,i}$, so we have direct sum decomposition $\tilde V_2=\oplus_j\tilde V_{2,I_j}$. For every $j\in \{1,\ldots,k\}$, let $K_j:=\pi_{\tilde V_{2,I_j}}(K)$.   
\begin{lemma}
\label{lem_dim_k_j_equals_1}
$\dim K_j=1$ for all $j\in \{1,\ldots,k\}$.
\end{lemma}
\begin{proof}
We first claim that $H$ preserves the collection $\{K_j\}_{1\leq j\leq k}$, and acts transitively on it.  To see this, note that $H$ preserves $K$ and the direct sum decomposition  $\tilde V_2=\oplus_j\tilde V_{2,I_j}$.  
   Therefore if $h\in H$ and $h(I_j)=I_{j'}$, then
       $h\circ\pi_{\tilde V_{2,I_j}}=\pi_{\tilde V_{2,I_{j'}}}\circ h$,  
     so $h(K_j)=K_{j'}$.  Since $H\acts \{I_1,\ldots,I_k\}$ is transitive, this proves the claim.

It follows that all the $K_js$ have the same dimension, and since $K\subset \oplus_jK_j$ we must have $\dim K_j\geq 1$ for all $j$. 

Now suppose $\dim K_j\geq 2$ for some $j$.  Choose $i_0\in I_j$.  Since $|I_j|\geq \dim K_j\geq 2$, the set $I_j\setminus\{i_0\}$ is nonempty, and it follows that $K_j$ intersects the hyperplane $\oplus_{i\in I_j\setminus\{i_0\}}\tilde V_{2,i}$ in some vector $W\neq 0$.  Let $J:=\{i\in I_j\mid \pi_{\tilde V_{2,i}}(W)\neq 0\}$.  Note that $J$ is invariant under $H'$, since $H'$ fixes $W$ and preserves the direct sum decomposition $\tilde V_2=\oplus_i \tilde V_{2,i}$.  Since $I_j$ is an $H'$-orbit, we must have $J=I_j$.   This contradicts the fact that $J\subset I_j\setminus\{i_0\}$. 
\end{proof}

\bigskip
Using an argument similar to Lemma~\ref{eq:dim1_diagonal},  up to a graded isomorphism, we may assume without loss of generality that each $K_j\subset V_{2,I_j}$ is ``diagonal'', i.e. $K_j=\Span(\sum_{i\in I_j}Y_i)$ for all $1\leq j\leq k$.

\begin{theorem}  \label{th:local_constancy_K_equals_2}
The permutation $\sigma_x$ induced by the Pansu differential $\tilde D_P f(x)$ is locally constant almost everywhere in $U$. 
\end{theorem}

\begin{proof}  
Recall that we have reduced to the case when $\tilde D_Pf(x)$ acts as the identity on $K$ and  $K_j=\Span(\sum_{i\in I_j}Y_i)$ for all $j\in \{1,\ldots k\}$. 
Thus $\tilde D_P f(x) \in H'$ for a.e. $x\in U$.   If $H' = \{ \id\}$ we are done.   So we assume that $H'\neq\{\id\}$ and will show that $\tilde D_Pf(x)$ acts as a constant permutation on each $H'$-orbit $I_j$, for $j\in \{1,\ldots,k\}$.
Then we can argue as in the case $\dim K=1$ to deduce constancy of the permutation.

We provide some details for the convenience of the reader. Let $I$ be an orbit and $\tilde V_{2,I} = \oplus_{i \in I} \tilde V_{2,i}$. The map $\tilde D_P f(x)$
preserves $\tilde V_{2,I}$ and acts as the identity on the subspace $K_I = \Span\{ \sum_{i \in I} \tilde Y_{2,i}\}$. 
Thus the argument in the proof of Proposition~\ref{le:auto_preserve_diagonal} shows that the identities in
 that proposition  hold with $\lambda = 1$. 
Recall that $P_{mi}(x) = \delta_{m \sigma_x^{-1}(i)}$.
Applying the pullback theorem to the forms $f_P^*\omega_{ij} \wedge d(\varphi i_X \beta)$ 
where  $\omega_{ij}$ and $\beta$ are  as in \eqref{eq:define_omega_ij} 
and \eqref{eq:define_beta_dim_1}  with $i, j, m\in I$  and where  $X  \in V_{1,l}$ with $l \ne m$ we get as in the proof of 
Lemma   \eqref{le:constant_permute_K_equals_1} the distributional identity
 $X P_{mi} = 0$ for all $m, i \in I$ and all $X  \in V_{1,l}$ with $l \ne m$.
Since $\sigma_x(I) = I$ we have $\sum_{m \in I} P_{mi} = 1$ for all $i \in I$. Now for $X \in V_{1,m}$
we get
$ X P_{mi} = X (1 - \sum_{m' \in I \setminus \{m\}} P_{m'i}) = 0$.
Thus $P_{mi}$ is constant for $m, i \in I$ and hence $\sigma_x|_{I}$ is constant. 
\end{proof}

\bigskip

\section{Product quotients of complex Heisenberg groups and higher real Heisenberg groups}
\label{sec_higher_product_quotients}
In this section we prove Theorem~\ref{thm_rigidity_product_quotient} for product quotients $G=\tilde G/\exp(K)$ where $\tilde G$ is a product of copies of a higher real Heisenberg group or complex Heisenberg group.  We retain the notation for product quotients from Section~\ref{sec_rigidity_product_quotients}.
 The approach used in this section is different, and conceptually
  simpler, than the ones used in Sections~\ref{sec_dim_k_equals_1} and \ref{sec_conformal_case}.  
The main difference between these two situations is that for the higher real and the complex Heisenberg group
there exist linearly   independent vectors $X,Y$ in $V_{1,i}$ with $[X,Y] = 0$ while this is not the
case for the first Heisenberg group.  More systematically, consider the map  
\begin{equation}
L_{[\, ]}: \Lambda_2 V_1 \to  V_2, \qquad
L_{[ \, ]}(X \wedge Y)= [X,Y].
\end{equation}
For the complex Heisenberg groups and higher real Heisenberg groups 
\begin{equation}
\label{eqn_ker_i}
\ker_i := \ker L_{[ \, ]} \cap V_{1,i}
\end{equation} 
is non-trivial (as follows from a dimension count) while for the first real  Heisenberg group
$L_{[ \, ]} : \Lambda_2 V_{1,i} \to     V_{2}$ is injective. 
 Any graded automorphism --  in particular the Pansu differential -- permutes the subspaces $\ker_i$. 
Now for complex or higher real Heisenberg product quotients  we can pull back a collection of forms $\alpha \in \Lambda^2 V_{1,i}$ which restrict to a basis of the dual space $\ker_i'$  and take $\beta = i_X i_Z \omega$ where 
$X \in \ker_k$ and $Z \in \oplus_{k' \ne k} V_{1,k'}$.  
The choice of this collection of forms $\alpha$ may look somewhat adhoc at this point, but it is in fact related to our
general strategy to look for closed forms modulo exact forms. The point is that 
the map $L_{[]}$ is the dual of the exterior differential 
  $d: \Lambda^1 V_2 \to \Lambda^2 V_1 \simeq (\Lambda_2 V_1)'$, see Remark~\ref{re:higher_dual}  after the end of the proof.
 With the above choice of $\alpha$ and $\beta$
 a short argument similar to the one used for products shows that the permutation of the subspaces
induced by $D_P f$ must be locally constant. The subspace $K$ plays no role in this argument. In fact the argument
can also be used to obtain an alternative proof for products of the complex or the higher real Heisenberg groups.

\bigskip
We now turn to the formal proof.
Note first that   \eqref{eq:exterior_codegree3},   \eqref{eq:exterior_codegree3bis},
the relation $[V_{1,i}, V_{1,j}] = 0$ for $i \ne j$ and \eqref{eq:i_circ_i}
imply that
\begin{equation}   \label{eq:exterior_codegree4}
d(i_{Y} i_{Y'} i_Z  \omega) =  - i_Z  i_{[Y,Y']} \omega \quad 
\hbox{$\forall \, \, Y, Y' \in  V_{1,i}, \, \, Z \in V_{1,j}$ with $i\ne j$.}
\end{equation}

\begin{lemma} \label{le:varphi_codegree3} Let $k \ne k'$,  $Y, Y'  \in V_{1,k}$,
$Z \in V_{1,k'}$ and $\gamma \in \oplus_{j=1}^n \Lambda^2(V_{1,j})$.
For $U \subset G$ open let $\varphi  \in C^1(U)$. 
Then the following identities hold in $U$: 
\begin{align}  \label{eq:ref_formula_d_eta}
d  (\varphi \,  i_Y i_{Y'} i_Z \omega) = & \, (Y \varphi) \, i_{Y'} i_{Z} \omega - \, (Y' \varphi) i_Y i_{Z} \omega + 
(Z\varphi)\,  i_{Y} i_{Y'} \omega\\
  + & \,  \varphi \, d ( i_Y i_{Y'} i_Z \omega),
    \nonumber 
\end{align}
\begin{equation}    \label{eq:wedge_d_eta}
\gamma \wedge d(\varphi \,  i_Y i_{Y'} i_Z \omega)  = -  \gamma(Y, Y')  \,  (Z \varphi) \,  \omega +
 \varphi \, \gamma \wedge d(i_Y i_{Y'} i_Z \omega).
\end{equation}
If $X \in \Lambda_2 V_{1,k}$ 
then
\begin{equation}  \label{eq:wedge_d_eta_bis}
\gamma \wedge d(\varphi \, i_X i_Z \omega) =  \gamma(X)  \,  (Z \varphi) \,  \omega + \varphi \, \gamma \wedge d(i_X i_Z \omega).
\end{equation}
\end{lemma}

\begin{proof} 
We  have
$$ d  (\varphi \,  i_Y i_{Y'} i_Z \omega)  - \varphi \, d (i_Y i_{Y'} i_Z \omega)
= d \varphi \wedge   i_Y i_{Y'} i_Z \omega.$$
Now the graded Leibniz rule \eqref{eq:Leibniz_interior}   for the interior derivative  gives
\begin{eqnarray*}
&  & d\varphi \wedge  i_Y i_{Y'} i_Z \omega \\
&=& (Y \varphi) i_{Y'} i_{Z} \omega -  i_Y (d \varphi \wedge i_{Y'} i_{Z} \omega) \\
&=& (Y \varphi) i_{Y'} i_{Z} \omega  - i_Y ((Y' \varphi) i_Z \omega) + i_Y i_{Y'} (d\varphi \wedge i_Z \omega)\\
&=& (Y \varphi) i_{Y'} i_{Z} \omega  - (Y' \varphi) i_Y i_Z \omega +(Z \varphi)  i_Y i_{Y'} \omega.
\end{eqnarray*}
In the last step we used \eqref{eq:inner_product_volume_form}.
This proves  \eqref{eq:ref_formula_d_eta}.

The  assertion    \eqref{eq:wedge_d_eta} now  follows from \eqref{eq:inner_product_volume_form}
and the fact that $\gamma(Y',Z) = \gamma(Y,Z) = 0$ since $Z$ lies in a different first layer subspace than $Y$ and $Y'$. Note that   \eqref{eq:i_circ_i}  gives the identity
$i_Y \circ i_{Y'}  = - i_{Y \wedge Y'}$ and thus $i_Y i_{Y'} \gamma = - \gamma(Y, Y')$.

Finally for $X = - Y \wedge Y'$  \eqref{eq:wedge_d_eta_bis} is the same
as   \eqref{eq:wedge_d_eta} and for general $X$  the identity  \eqref{eq:wedge_d_eta_bis}  follows by linearity.
\end{proof}

\bigskip

We now assume that $f:G\supset U\ra G$ is a $W^{1,p}_{\loc}$-mapping for some $p>\nu$ such that $D_Pf(x)$ is an isomorphism for a.e. $x\in U$. 
Recall that $\ker_i$ is defined in (\ref{eqn_ker_i}).  The main point  in the proof of Theorem~\ref{thm_rigidity_product_quotient} for product quotients of the higher groups
  is to show the following identity
\begin{align} \label{eq:key_identity_higher_groups}
&   \int_U \alpha((D_P f)_*X)  \, (Z \varphi) = 0   \qquad  
 \forall \varphi \in C_c^\infty(U), \\ & \forall   \alpha \in \oplus_{j=1}^n \Lambda^2 V_{1,j}, \, \, 
 X \in \ker_k   , \, \, Z \in V_{1,k'} \quad \hbox{with $k' \ne k$.} 
 \nonumber
 \end{align}
 Then we can deduce easily that the permutation induced by $D_P f$ is locally constant.

 To prove  \eqref{eq:key_identity_higher_groups} fix $\alpha$, $X$, $Z$ and $\varphi$ and consider the codegree $3$
 forms
  $$ \beta = i_{X \wedge Z} \omega \quad \hbox{and} \quad  \eta = \varphi \beta.$$
 It follows from   \eqref{eq:exterior_codegree4} and linearity that $\beta$ is closed. Moreover $\wt(\beta) = -\nu + 3$.
 Thus   Theorem~\ref{co:pull_back2} 
 implies that
\begin{equation} \label{eq:key_identity_higher_groups_bis}
 \int_U  f_P^*\alpha \wedge d \eta = 0.
 \end{equation}
Since $D_P f$ permutes the first layer subspaces $V_{1,j}$ it follows that 
$f_P^*\alpha  \in \oplus_{j=1}^n \Lambda^2 V_{1,j}$.
 Thus we get from  \eqref{eq:wedge_d_eta_bis}    and  \eqref{eq:exterior_codegree4}
$$ f_P^*\alpha \wedge d\eta =  f_P^*\alpha(X) \, (Z \varphi) \, \omega.$$
This yields \eqref{eq:key_identity_higher_groups}.

To deduce properties of the permutation $\sigma_x$ induced by $D_P f(x)$ recall that
$D_P f(x)$ maps $\ker_k  \subset \Lambda_2 V_{1,k}$ to 
$\ker_{\sigma_x(k)} $. 
Let 
$$ 
P_{kl}(x)= \delta_{\sigma_x(k) l}.$$

If we choose basis vectors $X_{k,m}$ of $\ker_k $ then there exist measurable functions $G_{mm'}: U \to \R$
such that
$$ (D_P f)_* X_{k,m}  = \sum_{l=1}^n  P_{kl} \sum_{m'} G_{m'm} X_{l,m'}$$
and the matrix $G(x)$ is invertible for a.e.\ $x$.  
Let $\alpha_{k,m'} \in  \Lambda^2 V_{1,k}$ be chosen such that the restrictions to $\ker_k$ 
yield a basis of the dual space $\ker_k'$ which is dual to the basis $X_{k,m}$, 
   i.e. $\alpha_{k,m'}(X_{k,m}) = \delta_{m m'}$ for each $k=1, \ldots, n$. 
Applying \eqref{eq:key_identity_higher_groups} with $X = X_{k,m}$ and $\alpha = \alpha_{l,m'}$
we get for each $k$ and $l$
\begin{equation}  \label{eq:key_identity_higher_groups_ter}
\int_U   P_{kl}  G   \, (Z \varphi) \, \omega = 0 \quad \forall \varphi \in C_c^\infty(U),  \, \, Z \in V_{1,k'} 
\quad \hbox{with $k' \ne k$}.
\end{equation}
Define a function $h$ on the space of matrices by $h(G) =1$ if $\det G \ne 0$ and $h(G) = 0$ else. 
Then $h$ is a Borel function. Since the matrix $G$ in  \eqref{eq:key_identity_higher_groups_ter}
is invertible a.e. in $U$, we have 
 $h(P_{kl} G) = P_{kl}$  a.e.\ and
it follows from Lemma~\ref{le:compact_directional_constancy}     that
\begin{equation}  \label{eq:key_identity_higher_groups_quart}
\int_U   P_{kl}   \, (Z \varphi) \, \omega = 0 \quad \forall \varphi \in C_c^\infty(U),  \, \, Z \in V_{1,k'} 
\quad \hbox{with $k' \ne k$}.
\end{equation}
In other words, 
\begin{equation}  \label{eq:permute_const_higher}
 Z P_{kl} = 0   \quad \forall  l, \quad \forall Z\in V_{1,k'}, \, \, k' \ne k  \, \, 
 \end{equation}
in the sense of distributions. 
By exchanging $k$ and $k'$ we get
\begin{equation}  \label{eq:permute_const_higher2}
 Z P_{k'l} = 0   \quad \forall  l, \quad \forall Z \in V_{1,k}, \, \, k' \ne k.  \, \, 
 \end{equation}
 Since $P_{kl} = 1 - \sum_{k' \ne k} P_{k'l}$ we deduce that $Z P_{kl} = 0$
 for all $Z \in V_1$ and all $k,l$. Thus $P_{kl}$ is locally constant and hence the permutation 
 of the first layer subspaces 
 induced by $D_Pf(x)$  is locally constant. This completes the proof.

\bigskip\bigskip

 \begin{remark}  \label{re:higher_dual} 
  The choice of the forms $\alpha_{k,l}$ can be motivated by our guiding principle to look for closed left invariant forms
 modulo exact left invariant forms  (see \cite[Lemma 4.8]{KMX1}) ). First note by   \eqref{eq:exterior_derivative_on_algebra} we have
 $d\alpha(X,Y) = - \alpha([X,Y])$ for $\alpha \in \Lambda^1 V_2$.
 If we identify two-forms on $V_1$ with the dual space of two-vectors by setting $\alpha(X \wedge Y) = \alpha(X,Y)$,
 then  we see  that exterior differentiation $d: \Lambda^1 V_2 \to (\Lambda_2 V_1)'$
  is dual map of $- L_{[]}$. Now consider first the case that $\fg = \oplus_i \fh_i$ is a direct sum of $n$ copies of the same
  (higher) Heisenberg group. A natural space of closed left invariant forms to detect the permutation $\sigma_x$ induced by
  the Pansu differential is the space $\oplus_i \Lambda^2 V_{1,i}$. 
   In this case the relevant space of exact forms is given by $d \Lambda^1 V_2 = \oplus_i d \Lambda^1 V_{2,i}$.
  By the duality between $d$ and $-L_{[]}$ all elements of $d\Lambda^1 V_2$ vanish on all the space $\ker_i$ and the sets
  $\ker_i$ are characterized by this condition. Thus to find a basis for the quotient space $\oplus_i  \Lambda^2 V_{1,i}/ d\Lambda^1 V_2$
  it is natural to look for elements of  $\Lambda^2 V_{1,i}$ whose restriction to $\ker_i$  does not vanish. This leads to the basis $\alpha_{k,l}$
  used in the proof.
  The fact that  there exist  non-zero $X \in \ker_i$  is crucial  for finding {\it closed} codegree  forms $\beta = i_Z i_X \omega$
  which interact well with the forms $\alpha_{k,l}$.

 If $\fg$ is a product quotient $\oplus_i \tilde\fg_i / K$
  then it is still true that each element of $d\Lambda^1 V_2$ vanishes on each space $\ker_i$. Hence the collection of cosets
  $\alpha_{k,l} + d\Lambda^1 V_2$  is still linearly independent  in the quotient space and the proof shows that it is sufficiently rich to detect the permutation
  $\sigma_x$.
   \end{remark}

\bigskip

\bibliography{product_quotient}
\bibliographystyle{amsalpha}

\end{document}